\documentclass{amsart}
\usepackage[foot]{amsaddr}
\usepackage{amssymb,latexsym,amsmath,amsfonts,graphicx,xcolor,pb-diagram}
\usepackage[mathscr]{eucal}
\usepackage{tikz, tikz-cd}
\usepackage{relsize}
\usepackage{slashed}

\newcommand{\mpar}{\par \medskip \par }

\numberwithin{equation}{section}

\swapnumbers

\theoremstyle{definition}

\newtheorem{theorem}{Theorem}[section]

\newtheorem{corollary}[theorem]{Corollary}

\newtheorem{definition}[theorem]{Definition}

\newtheorem{remark}[theorem]{Remark}
\newtheorem{example}[theorem]{Example}

\makeindex

\begin{document}
\title{A Simple Model of Sea Spray}
\author{Samantha Donner$^1$, Peter March$^2$, Mathew Umano$^1$}

\begin{abstract}
We propose a stochastic model of sea spray as a shot noise process attached to a homogeneous Poisson point process whose intensity measure is $\mathcal{F}(r) dr dt,$ where $\mathcal{F}(r)$ is the sea spray generation function, that is to say the average number of droplets of radius $r$ ejected into the atmosphere per unit area per unit time. We assume the impulse response function defining the shot noise is determined by the trajectory $z=z_t(r)$ of the center of a spherical droplet of radius $r=r_t.$ We also assume  the droplet trajectory is vertical and droplets don't interact with one another. Under these assumptions, for each triple $(\mathcal{F}, z, r)$ we show the total mass of airborne sea spray droplets below height $z$ at time $t$ is a stationary stochastic process $\mathcal{S}_t(z)$ whose mean $m(z)$ and covariance $\sigma(z,t; w,s)$ can be written explicitly in terms of $z_t(r), r_t$ and $\mathcal{F}(r).$ We illustrate the model through examples of increasing physical relevance, the most general of which includes gravity, Stokes drag forces, and Langmuir's $D$-squared law of evaporation. A consequence of this particular example is an estimate of the height of the marine boundary layer at sufficiently large wind speeds.
\end{abstract}

\address{$^1$Department of Environmental Sciences, Rutgers University - New Brunswick} 
\address{$^2$Department of Mathematics, Rutgers University - New Brunswick}

\maketitle
\section{Introduction}
Breaking waves on the ocean create air bubbles in the surface layer which, when they burst, eject droplets of sea water into the atmosphere creating sea spray. These droplets subsequently follow dynamical trajectories governed by Newton's laws and can persist in the atmosphere for significant amounts of time. Because of its ubiquity and persistence, sea spray is understood to play an important role in the physics and chemistry of the marine boundary layer, e.g. \cite{V}, \cite{RV} The complexity of the underlying physical processes naturally leads to modeling sea spray in stochastic terms.  Our goal is to create a simple model of sea spray that is physically plausible yet mathematically tractable. 

\mpar
We make the stochastic nature of sea spray explicit by introducing ideas and techniques of probability theory to describe the random distribution of airborne droplets in time and in elevation above sea level. Assuming ejected droplets follow trajectories governed by physical law, and assuming that moving droplets do not interfere with one another, the number of droplets at location $z$ at time $t$ is the sum over all droplets ejected at times $s<t$ whose trajectory at time $t-s$ places them at $z.$ Physical processes of this kind are typically modeled as \textit{shot noise} in which a random event occurring at one time produces an after effect at a subsequent time. 

\mpar
We show the sea spray generation function is the intensity measure of a Poisson point process that counts the number of spherical droplets of a given range of radii ejected in a given interval of time. We assume the Poisson process is time homogeneous, which is equivalent to assuming equilibrium sea state conditions. To make things tractable, we simplify the physics of droplet motion by assuming an ejected droplet of initial radius $r$ is a sphere of radius $r_t$ centered at $z_t(r),$ which we think of a point particle carrying an amount of mass proportional to the cube of its radius. We assume its trajectory is purely vertical and it does not interfere with other droplets. The pair $(z_t(r), r_t)$  is used to define the \textit{impulse response function} and hence the corresponding shot noise $\mathcal{S}_t,$ which we call the \textit{sea spray} process. 

\mpar
In technical terms, $\mathcal{S}_t$ is a measure-valued process. In particular, this means for any interval $(a,b],$ where $0<a<b<\infty$ are specified heights above sea level, $\mathcal{S}_t((a,b])$ is the total mass at time $t$ of airborne sea spray droplets at height $z$ for any $a<z\leq b.$ Let $\mathcal{S}_t(z) =\mathcal{S}_t((0,z])$ denote the total mass of airborne sea spray at or below height $z$ at time $t.$ This is a random variable for each $t$ and $z.$ It develops that the two parameter process $(t,z)\to\mathcal{S}_t(z)$ is the object of practical interest. Although it is a secondary quantity derived from $\mathcal{S}_t,$ we refer to either of them as the sea spray process. 

\mpar
Under mild technical assumptions we show that for each triple $(\mathcal{F}, z, r)$ there is a unique stationary stochastic process $\mathcal{S}_t(z)$ which measures the total mass of sea spray below height $z.$  We calculate its mean function,
$$
m(z)=E(\mathcal{S}_t(z))
$$
and covariance function,
$$
\sigma^2(z,t; w,s) =E(\,[\mathcal{S}_t(z)-m(z)]\cdot [\mathcal{S}_s(w)-m(w)\,)
$$ 

as explicit integrals involving $z_t(r)$ and $r_t$ with respect to the measure $\mathcal{F}(r)dr dt.$

\mpar
We use the mean function to define quantiles $z_X,\, 0\leq X\leq 100,$ such that on average $X\%$ of sea spray mass lies below height $z_X.$ We propose a definition of the height of the marine boundary layer, $z_{MBL},$ as the top quantile, $z_{100},$ and show that for typical sea spray generation functions, $z_{100}$ is a function of the trajectory alone. Specifically, $z_{100}$ equals the maximum height reached by any droplet, which we denote by $z^{**}.$ That is to say the three quantities $z_{MBL}=z_{100}=z^{**}$ are equal. (See Corollary 3.5).

\mpar
The mathematical theory does not depend on the precise nature of $z_t(r),\, r_t,$ or $\mathcal{F}(r),$ so it applies to a broad range of possible sea spray generation functions and physical trajectory functions. To illustrate the model, we consider several examples of increasing physical relevance. 

\mpar
In Example 4.1 we calculate $m(z), \sigma(z,t;w,s),$  and $z_X$ exactly when $\mathcal{F}(r)$ is constant on a finite interval $0<r_{min}\leq r\leq r_{max}<\infty$ and $z_t(r)=z_t$ is the parabolic trajectory of a freely falling particle of initial velocity $v.$ 

\mpar
In Example 4.2 we include a vertical wind field and Stokes drag forces. The solution depends on $\rho,$ the mass density of seawater, $\mu,$ the dynamic viscosity of the atmosphere, $v,$ the initial velocity, and $w,$ the vertical wind speed. For typical $\mathcal{F}$ we find (1) for sufficiently small wind speeds $m(z)$ is bounded for all values of $z$, leading to a finite marine boundary layer, (2) there is an intermediate range of wind speeds such that $m(z)=\infty$ for all sufficiently large $z$, and (3) for sufficiently large wind speeds, $m(z)$ is finite but unbounded, meaning the marine boundary layer is infinite.

\mpar
The unreasonable features of Example 4.2 suggest the absence of important physical processes, such as heat and mass transfer. In Example 4.3 we adopt the simplest model of evaporation, namely Langmuir's $D$-squared law \cite{L}, which states that the surface area of a spherical droplet or, equivalently the square of its diameter, decreases linearly. This is equivalent to,
$$
r^2_t=r_0^2-Kt,\quad 0\leq t\leq r_0^2/K
$$
for some $K>0,$ called the evaporation constant, which has units of $length ^2/time.$

\mpar
We derive formulas for $z_t(r)$ which show precisely how the trajectory depends on physical parameters. The formulas are explicit but too involved to derive exact formulas for $m(z)$ and $\sigma(z,t;w,s),$ although it's feasible to compute them numerically for physically reasonable $\mathcal{F}.$ However, we can show that for general $\mathcal{F}$ the mean $m(z)$ is a bounded function of $z,$ which shows, in particular, that the height of the marine boundary layer, $z_{MBL},$  is finite. For positive, continuous $\mathcal{F}(r)$ supported on a finite interval $r_{min}\leq r\leq r_{max}$ we also show that,
$$
\lim_{w\to\infty}\frac{z^{**}}{w}=\frac{9\mu}{9\mu +2\rho K}\times\frac{r^2_{max}}{K},
$$
where $z^{**}$ is the maximum height reached by any droplet. This implies a rule-of-thumb approximation,
$$
z_{MBL} \approx \frac{9\mu}{9\mu +2\rho K}\times w\times\frac{r_{max}^2}{K} 
$$
for the height of the marine boundary layer at large wind speeds.

\mpar
The intuition behind this approximation is pretty clear. At sufficiently large wind speeds no droplet returns to the ocean's surface before evaporating. But before evaporating completely, every droplet's radius reduces to a sufficiently small size that it acts ballistically and is advected by the wind. Since larger droplets take longer to evaporate, droplets of maximum radius are advected for the greatest length of time. So the approximation just says that the maximum height attained by any droplet is proportional to wind speed multiplied by the evaporation time of the largest droplet. The proportionality constant is less than one, to account for transitory initial effects, and depends explicitly on physical parameters.  See Example 4.3 for an exact determination of $z^{**}$ for sufficiently large wind speeds, which is mathematically rigorous but less intuitive than the argument given here.

\mpar\mpar 
\section{Physical Assumptions}
We assume the dynamic sea state conditions are in equilibrium, meaning the physical processes generating waves and bursting bubbles are also in equilibrium. We focus on a small patch of the ocean's surface which is large enough so that the number of ejected droplets per unit time has some statistical regularity. A conventional choice is a square of side length equal to a centimeter, which we call a unit patch. We assume that droplet ejection events in disjoint unit patches are statistically independent and identically distributed. The mean number of droplets of radius $r$ ejected per square centimeter per second is the sea spray generation function $\mathcal{F}(r)$. 

\mpar
There are different modes of droplet production, for example jet drops or film drops, and each mode leads to its own generation function, $\mathcal{F}_{jet}$ or $\mathcal{F}_{film}.$ Since the physical mechanisms in each production mode are assumed to be independent of one another, a general sea spray generation function is understood be be a sum of several generation functions, one for each principal mode of production.

\mpar
The sea spray generation function proposed by Dieke, Reichi, and Paulot \cite{DRP} has the functional form,
$$
\mathcal{F}(r)=\frac{V_A}{\mathcal{C}_A}\, \mathcal{A}\, r^{-(1+\zeta)}\left[ \Gamma_{inc}\left(m+\zeta, m\,\tfrac{r}{r_{max}}\right) - \Gamma_{inc}\left(m+\zeta, m\,\tfrac{r}{r_{min}}\right)\right].
$$
Here, $\Gamma_{inc}(s,x)$ is the incomplete gamma function and $V_A, \mathcal{C}_A$ and $\mathcal{A}$ are normalizing constants. Upper and lower bounds on radii are denoted by $r_{min}, r_{max}.$ They, along with the other parameters $\zeta$ and $ m$ depend on the production mode. However, for the purposes of this paper, all we require is that $\mathcal{F}(r)$ is a positive, continuous function on a finte interval $[r_{min}, r_{max}].$ 
\mpar
Next, let's consider what happens to a droplet once it's been ejected into the atmosphere. We assume the atmosphere is spatially homogeneous and subject to a vertical wind field of constant speed. We assume that a droplet of sea spray only moves up and down above the ocean patch. Droplets are not actually moving in a perfect vertical line but rather in a cylinder with a square centimeter base. It's reasonable to assume the linear dimension of a droplet is significantly smaller than a centimeter and to assume the density of droplets per unit volume is sufficiently small that droplets do not collide but rather move independently of one another.

\mpar
Finally, we assume that droplets are spheres of various radii and we idealize a droplet as a pair $(r_t, z_t(r))$ specifying a point particle located at the center of the sphere that carries all the mass of the droplet and we think of the pair as being determined by physical law.

\mpar
Of course there are a number of ways to specify $z_t$ exactly, depending on the physics.  The most general example we consider includes gravity, Stokes' approximation for the drag force that the wind exerts on the droplet, and Langmuir's $D$-squared law of evaporation. Given these assumptions, $z_t(r)$ satisfies the following pair of equations so long as $r_t>0$ and  $z_t(r) >0,$
\begin{align*}
\tfrac{4}{3}\pi r_t^3 \rho z_t^{\prime\prime}(r) & = -\tfrac{4}{3}\pi r_t^3\rho g +6\pi\mu r (w-z_t^\prime(r)),\quad z_0(r)=0,\quad z_0^\prime(r) =v\\
r^2_t & = r^2-Kt,\quad 0\leq t <r^2/K,
\end{align*}

where $\rho$ is the mass density of sea water, $g$ is the force of gravity, $\mu$ is the dynamical viscosity of the atmosphere, $w$ is the vertical wind speed, $v$ is the initial velocity, and $K$ is the evaporation constant.

\mpar
We are able to calculate the trajectory exactly in Example 4.3, but more physically realistic examples, say where the evaporation rate depends on the droplet's position and velocity, will require numerical approximation. Nevertheless, all our model requires is a choice of sea spray generation function $\mathcal{F}(r)$ to determine the average rate at which droplets of a given radius are ejected, a choice of dynamical trajectory $z_t(r)$ to determine the height of the droplet above sea level as a function of time, and a choice of radius $r_t$ to determine the mass of the droplet.  The more physically accurate are these choices, the more relevant is the model.

\mpar\mpar
\section{Mathematical Theory}
In this section we neglect the physical accuracy of $\mathcal{F}(r),\, z_t(r),$ and $r_t$ and simply focus on their mathematical aspects. From this perspective, the sea spray process is a function of the inputs $(\mathcal{F}, z, r).$

\mpar
Let $\mathcal{F}(dr)$ be a nonnegative measure on the positive half line $(0,\infty),$ which we think of as the set of all possible droplet radii. Since $\mathcal{F}(dr)$ represents the number of droplets with radii in the infinitesimal interval between $r$ and $r+dr$ that are ejected per square centimeter per second, we assume $\mathcal{F}(dr)$ is a finite measure to ensure the total rate of droplet ejection is finite. Let,

$$
\Lambda(dr\, dt)= \mathcal{F}(dr) dt 
$$

\mpar
be the associated measure on $\mathcal{X}=(0,\infty)\times\mathbb{R},$ where we think of a point $(r,t)\in \mathcal{X}$ as marking a time at which a droplet of radius $r$ is ejected from the sea surface. 

\mpar
We interpret $\Lambda$ as the intensity measure of a Poisson point process that counts the random number of ejected droplets of radius $r$ at time $t.$ By a standard theorem in probability \cite{K}, there exits a unique stochastic process $\mathcal{N}$ with the following three properties:

\mpar\mpar
\begin{itemize}
\item{For every $n\geq 1,$ let $R_1, R_2, \cdots R_n \subset (0,\infty)$ be finite intervals of possible droplet radii and let $T_1, T_2, \cdots , T_n\subset\mathbb{R}$ be finite intervals of possible ejection times. Then,

$$
\mathcal{N}(R_1\times T_1),\,\,\, \mathcal{N}(R_2\times T_1),\,\,\, \cdots,\,\,\, \mathcal{N}(R_n\times T_n)
$$

\mpar
\negthickspace\negthickspace\negthickspace\negthickspace are integer valued random variables, and they are statistically independent provided $R_k\times T_k,\,\,1\leq k\leq n,$ are pairwise disjoint subsets of $\mathcal{X}$ }\\

\item{For every $u\in\mathbb{R},$ let $T_u=\{t+u\mid t\in T\}$ be the translation of $T$ by amount $u.$ Then, $\mathcal{N}(T_u\times R)$ and $\mathcal{N}(T\times R)$ are identically distributed for every interval $R$}\\

\item{ For every pair of finite intervals $R$ and $T$,

$$
P(\mathcal{N}(R\times T)=k) =\frac{(\Lambda(R\times T))^k}{k!} e^{-\Lambda(R\times T)}
$$}

\end{itemize}

\mpar
The third bullet says that if $R=(a,b],\,T= (s, t],$ and $\mathcal{F}$ has a density then the random number of droplets $\mathcal{N}(\,(a,b]\times (s,t]\,)$ ejected at times $s<u\leq t$ with radii $a<r\leq b$ follows a Poisson distribution of intensity,

$$
\Lambda(R\times T)= \int_a^b\int_s^t\mathcal{F}(dr) du =(t-s)\negthickspace\int_a^b\mathcal{F}(r)dr.
$$

\mpar
The second bullet says that the counting process $\mathcal{N}$ is time homogeneous which reflects the physical assumption of equilibrium sea state conditions. The first bullet says that event counts on disjoint subsets of $\mathcal{X}$ are statistically independent. 

\mpar
This is the sense in which the sea spray generation function $\mathcal{F}$ determines a unique homogeneous Poisson point process $\mathcal{N}$.

\mpar
We create a shot noise process from the counting process $\mathcal{N}$ and a  family of trajectories $z_t(r)$ and their associated droplet radii $r_t$ by introducing the impulse response function,
$$
\phi(t,r)=\tfrac{4}{3}\pi r_t^3\rho\,\delta_{z_t(r)},\quad r_0=r,\,\, z_0=0
$$

where $\delta_z$ is the Dirac measure centered at $z\geq 0.$ Evidently, $\phi(t,r)$ records the fact that if a droplet of radius $r$ was ejected at time 0, then there is a volume $\tfrac{4}{3} \pi r_t^3$ of sea water, hence a mass of $\tfrac{4}{3} \pi r_t^3\rho,$ at height $z_t(r)$ at time $t$.

\mpar
Different physical assumptions will lead to different trajectories $z_t(r)$ and radius functions $r_t$. In mathematical terms, all we require are the following conditions. Let,
$$
\zeta=\zeta(r)= \inf\{\, t>0\mid z_t=0\quad\text{or}\quad r_t=0\,\}.
$$
be the \textit{duration} of the droplet. Note that the pair $(z_t, r_t)$ is indexed by $r$ which means $\zeta$ is a function of $r.$  We assume,
\mpar
\begin{itemize}
\item{$r_t, \,z_t(r)$ are continuous, nonnegative functions of $t$} on the interval $[0,\zeta]$
\item{if $r_t=0$ for some $t>0$ then $r_s=0$ and $z_s=0$ for all $s\geq t$}
\item{if $z_t(r)=0$ for some $t>0$ then $z_s(r)=0$ and $r_s=0$ for all $s\geq t$}
\end{itemize}

\mpar
The third bullet states that droplets are \textit{inelastic}: droplets which return to the ocean merge back into to ocean as opposed, say, to bouncing off the surface back into the atmosphere. Bullet two states that if a droplet evaporates then it can't condense into a new droplet at a later time. Furthermore, for purely technical reasons, if a droplet has no mass then its height is reset to zero. Bullet one is an important technical assumption which is roughly equivalent to assuming that the droplet's radius and trajectory are continuous functions of physical parameters.

\mpar
In intuitive terms, the sea spray process  is the mass distribution at time $t$ of airborne sea spray droplets, which means that $\mathcal{S}_t$ is a random measure on $(0,\infty).$ To put this on a firm mathematical foundation, let $\mathcal{M}$ be the space of nonnegative measures on the vertical half line $[0,\infty)$ and let $C([0,\infty),\mathcal{M})$ be the space of all continuous functions from $[0,\infty)$ to $\mathcal{M}.$ It's clear that $\phi(r,t)\in\mathcal{M}$ for every $r>0$ and $t\geq 0$ and therefore for each fixed $r$ the function $t\to\phi(t,r),\, t\geq 0$ is an element of $C([0,\infty), \mathcal{M}).$ Finally, thinking of $0<r<\infty$ as a variable, we see that the formula for $\phi$ implies,
$$
\phi\colon (0,\infty)\to C([0,\infty), \mathcal{M})
$$
is a function from the space of initial radii to the space of continuous functions from the nonnegative time line to the space measures on the nonnegative space line.

\mpar
Let's fix a time $t$ and ask which pairs $(s, r),$ or \textit{bursts}, can contribute to $\mathcal{S}_t$? Evidently no burst with $s>t$ contributes since it hasn't happened yet and any burst with $s \leq t$ may possibly contribute something. Thus, a burst $(s,r)$ contributes mass at a positive height if and only if $t-s> 0$ and $z_{t-s}(r)> 0.$  

\mpar
This observation just amounts to the statement that the sea spray process is the integral of all contributing bursts with respect to the Poisson point process, namely,
$$
\mathcal{S}_t=\int_{-\infty}^t\int_0^\infty\phi(t-s,r)\, \mathcal{N}(dr, ds).
$$

Let's state this formally for the record.
\begin{definition}
Let $\mathcal{F}(dr)$ be a finite measure on $(0,\infty)$ and for each fixed $r>0,$ let $(z_t(r), r_t) $ be a pair of nonnegative functions of $t\geq 0.$ Let
$$
\zeta=\inf\{\, t>0\mid z_t=0\quad\text{or}\quad r_t=0\,\}.
$$
be the duration of the droplet. We assume, 
\begin{itemize}
\item{$r_t, \,z_t(r)$ are continuous  on the interval $[0,\zeta]$}
\item{if $r_t=0$ for some $t>0$ then $r_s=0$ and $z_s=0$ for all $s\geq t$}
\item{if $z_t(r)=0$ for some $t>0$ then $z_s(r)=0$ and $r_s=0$ for all $s\geq t$}
\end{itemize}

\mpar
Let $\mathcal{X}=\mathbb{R}\times (0,\infty)$ and let $\Lambda(dr\, dt)=\mathcal{F}(dr) dt.$ Let $\mathcal{N}$ be the Poisson point process whose intensity measure is $\Lambda.$ Let $\mathcal{M}$ be the space of positive measures on $[0,\infty)$ and let $C([0,\infty),\mathcal{M})$ be the space of continuous functions from $[0,\infty)$ to $\mathcal{M}.$

\mpar
Let $\phi\colon (0,\infty)\to C([0,\infty), \mathcal{M})$ be defined by the rule,
$$
\phi(r,t) = \tfrac{4}{3}\pi r_t^3\rho\, \delta_{z_t(r)}
$$
Then the sea spray process is defined by the formula,

$$
\mathcal{S}_t =\int_{-\infty}^t\int_0^\infty \phi(t-s,r)\, \mathcal{N}(dr\, ds).
$$
\end{definition}

\mpar
\begin{remark}
\textit{(1) Observe that for each fixed $t,\,\mathcal{S}_t$ is a random variable with values in $\mathcal{M}$. By construction it is a random sum of Dirac measures at points $z\in [0,\infty).$ }

\mpar
\textit{(2) No droplet makes a contribution to the sea spay process after its duration.}

\mpar
\textit{(3) All integrals defining the sea spray process are integrals of nonnegative functions with respect to nonnegative measures. The mathematical theory we use permits infinite values for such integrals but generally speaking they only occur in unusual circumstances. For example, if the trajectory function permits suspended droplets for which $a\leq z_t(r)\leq b$ for all sufficiently large $t$ or if $r_t$ diverges at a finite time, neither of which is physically realistic.}
\end{remark}

\mpar
By construction, $\mathcal{S}_t$ is a random sum of Dirac measures on the vertical half line $[0,\infty).$ We are interested in the mass distribution of airborne droplets as a function of time so we want to take care to omit the initial mass at 0. To do that rigorously, let $z>0$ be a height, let $0<\epsilon < z,$ and consider the $\mathcal{S}_t$ measure of the set $[\epsilon ,z],$ namely $\mathcal{S}_t([\epsilon, z]).$ The limit,
\begin{align*}
\mathcal{S}_t(z) & = \lim_{\epsilon\to 0} \mathcal{S}_t([\epsilon, z])\\
& = \tfrac{4}{3}\pi\rho \int_{-\infty}^t\int_0^\infty r_{t-s}^3 1_{(0, z]}(z_{t-s}(r))\,\mathcal{N}(dr\, ds).
\end{align*}

\mpar
is precisely the total  mass of airborne droplets at height less than or equal to $z$ at time $t.$ 

\mpar
\begin{remark}
\textit{In practical terms, the two parameter process  $\mathcal{S}_t(z)$ is the quantity of interest and we will refer to both $\mathcal{S}_t(z)$ and $\mathcal{S}_t$ as the sea spray process. This is ambiguous terminology but context usually makes usage clear. In any event, the formula,
$$
\mathcal{S}_t((z, z^\prime]) = \mathcal{S}_t(z^\prime)-\mathcal{S}_t(z)
$$
shows that the random function $\mathcal{S}(z),\, 0< z <\infty,$ determines the random measure $\mathcal{S}_t$ away from $z=0,$ so there really isn't any ambiguity. }
\end{remark}

\mpar 
Key quantities of interest regarding a stationary stochastic process are its mean and covariance. Since $\mathcal{S}_t(r)$ is derived from an underlying Poisson point process, there are explicit formulas for these characteristics.
\begin{theorem}
The sea spray process $\mathcal{S}_t(z)$ is a stationary stochastic process with mean value function,
\begin{align*}
m(z) & = E\left[\mathcal{S}_t(z)\right]\\
& =\tfrac{4}{3}\pi\rho   \int_0^\infty \negthickspace \int_0^\infty r^3_s\, 1_{(0,z]}(z_s(r))\mathcal{F}(dr) ds,
\end{align*}

and covariance function,
\begin{align*}
& \sigma^2(t,z; s, w)  = E(\,[\mathcal{S}_t(z)-m(z)]\cdot [\mathcal{S}_s(w)-m(w]\,)\\
& = \tfrac{16}{9}\pi^2\rho^2 \int_0^\infty \int_0^\infty r^3_{t -s\land t +u} r^3_{s -s\land t +u}\, 1_{(0, z]}(z_{t -s\land t +u}(r))\,1_{(0, w]}(z_{s-s\land t +u}(r))\mathcal{F}(dr) du,
\end{align*}
where $a\land b$ is the smaller of $a$ and $b$ and $1_A$ is the indicator function of the set $A.$ In particular, the equal-time covariance is,
$$
\sigma^2(z,t; w,t) = \tfrac{16}{9}\pi^2\rho^2 \int_0^\infty \int_0^\infty r^6_u\, 1_{(0, z\land w]}(z_u(r))\mathcal{F}(dr) du.
$$
\end{theorem}

\begin{proof}
Stationarity of $\mathcal{S}_t(z)$ follows from the time homogeneity of $\mathcal{N}$. Recall that for every $u\in\mathbb{R}$ and intervals $R$ and $T,$ the random variables $\mathcal{N}(T\times R)$ and $\mathcal{N}(T_u\times R)$ are identically distributed. Defining a new Poisson point process $\theta_u\,\mathcal{N}$ by the rule,
$$
\theta_u \,\mathcal{N}(T\times R)=\mathcal {N}(T_u\times R),
$$

it follows that  $\theta_u\,\mathcal{N}$ and $\mathcal{N}$ are identically distributed, meaning each of them is a Poisson point process with intensity measure $\mathcal{F}(dr) dt.$ Thus, changing the time variable we have,
\begin{align*}
\mathcal{S}_{t+s} & = \tfrac{4}{3}\pi\rho \int_{-\infty}^{t+s}\int_0^\infty r^3_{t+s-u} 1_{(0, z]}(z_{t+s -u}(r))\,\mathcal{N}(dr\, du)\\
& = \tfrac{4}{3}\pi\rho \int_{-\infty}^t\int_0^\infty r^3_{t-u} 1_{(0, z]}(z_{t-u}(r))\,\theta_s\,\mathcal{N}(dr\, ds) \cong \mathcal{S}_t
\end{align*}
which proves stationarity.

\mpar
Formulas for the mean and covariance are straightforward applications of \textit{Campbell's theorem} \cite{K} which relates expected values of integrals with respect to a point process to integrals with respect to its intensity measure. When applied to the Poisson point process $\mathcal{N},$  Campbell's theorem states the following. Let $f,g\colon\mathcal{X}\to\mathbb{R}_+$ be nonnegative functions and let,
$$
\mathcal{N}(f)=\int_{-\infty}^\infty\int_0^\infty f(r,t)\mathcal{N}(dr\,dt)\quad\text{and}\quad \Lambda(f)=\int_{-\infty}^\infty\int_0^\infty f(r,t)\Lambda(dr\,dt).
$$

Then,
$$
E\left[\mathcal{N}(f)\right] = \Lambda(f)\quad\text{and}\quad E\left[\mathcal{N}(f)\mathcal{N}(g)\right] =\Lambda(f)\Lambda(g) +\Lambda(fg).
$$

\mpar
In this notation, $\mathcal{S}_t(z)=\mathcal{N}(f_t)$ where,
$$
f_t(r,u) = \tfrac{4}{3}\pi r^3_{t-u}\rho\, 1_{(0, z]}(z_{t-u}(r))\,1_{(-\infty, t]}(u).
$$
By stationarity, the mean function $m(z)$ does not depend on $t.$ Thus $m(z) = E\,\left[\mathcal{S}_0\right]$ so this is just Campbell's theorem applied to,
$$
f_0(r,u)  =  \tfrac{4}{3}\pi r_{-u}^3\rho\, 1_{(0,z]}(z_{-u}(r))\,1_{(-\infty,0]}(u).
$$
To see the details note that,
\begin{align*}
E\,\left[\mathcal{S}_t\right] & =  E\,\left[\mathcal{N}(f_t)\right]\\
& = \tfrac{4}{3}\pi\rho E\, [\int_{-\infty}^t\int_0^\infty r^3_{t-u} 1_{(0, z]}(z_{t-u}(r))\,\mathcal{N}(dr\, du)]\\
& = \tfrac{4}{3}\pi\rho E\,[\int_{-\infty}^0\int_0^\infty r^3_{-u} 1_{(0, z]}(z_{-u}(r))\,\theta_t\mathcal{N}(dr\, du)]\\
& =E[\theta_t\mathcal{N}(f_0)] = E[\mathcal{N}(f_0)]=\Lambda(f_0)\\
& = \tfrac{4}{3}\pi\rho   \int_{-\infty}^0  \int_0^\infty r^3\, 1_{(0,z]}(z_{-u}(r))\mathcal{F}(dr) ds
\end{align*}

\mpar
Turning to the covariance function $\sigma^2(t, z; s, w)$ let,
$$
f_t(r,u)=\tfrac{4}{3}\pi\rho\, r^3_{t-u}\, 1_{(0,z]}(z_{t-u}(r))\,1_{(-\infty, t]}(u)
$$
and
$$
g_s(r,u)=\tfrac{4}{3}\pi\rho\, r^3_{s-u}\, 1_{(0,w]}(z_{s-u}(r))\,1_{(-\infty, s])}(u).
$$

\mpar
Then $\mathcal{S}_t(z)=\mathcal{N}(f_t)$ and $\mathcal{S}_s(w) =\mathcal{N}(g_s).$ In this notation we have,

\begin{align*}
Cov\left(\mathcal{S}_t(z), \mathcal{S}_s(w)\right) & = E\left[(\mathcal{N}(f_t)-m(z))\cdot(\mathcal{N}(g_s)-m(w))\right]\\
& \negthickspace\negthickspace\negthickspace\negthickspace\negthickspace\negthickspace\negthickspace\negthickspace\negthickspace  = E\left[\mathcal{N}(f_t)\mathcal{N}(g_s)\right] -m(z)E\left[\mathcal{N}(g_s)\right] - m(w) E\left[\mathcal{N}(f_t)\right] +m(z)m(w).
\end{align*}

\mpar
Note that, $E\left[\mathcal{N}(f_t)\right] =m(z)=\Lambda(f_t)$ and $E\left[\mathcal{N}(g_s)\right] =m(w) =\Lambda(g_s)$, by stationarity. Thus, by Campbell's theorem,
\begin{align*}
Cov\left(\mathcal{S}_t(z), \mathcal{S}_s(w)\right) & =  E\left[\mathcal{N}(f_t)\mathcal{N}(g_s)\right] -m(z) m(w)\\
& = \Lambda(f_t)\Lambda (g_s) +\Lambda(f_tg_s) - m(z)m(w)\\
& = \Lambda(f_tg_s).
\end{align*}
But,
$$
f_t(r,u)g_s(r,u)=\tfrac{16}{9}\pi^2 \rho^2r^3_{t-u}r^3_{s-u}\, 1_{(0, z]}(z_{t -u}(r))1_{(0, w]}(z_{s-u}(r)) 1_{(-\infty, s\land t]}(u)
$$

\mpar
and the change of variables $v=s\land t-u$ in the time integral finishes the proof.
\end{proof}

\mpar
In a reasonable theory, the mean function $m(z)$ is finite for all values of $z$ and bounded, meaning there exists a height $z_{100}$ such that $m(z)=m(z_{100})$ for all $z\geq z_{100}$ and $m(z)< m(z_{100})$ for some $z<z_{100}.$ Assuming $m(z)$ is continuous and strictly increasing on the interval $[0, z_{100}],$ we can define quantiles $z_X$ as the unique heights such that $X\%$ of the average total mass of sea spray lies below $z_X$, that is,
$$
\frac{m(z_X)}{m(z_{100})}=\frac{X}{100},\quad 0\leq X\leq 100.
$$ 

Thus, quantiles show precisely how sea spray is distributed in the atmosphere. In particular the top quantile, $z_{100},$ quantifies the height of the marine boundary layer in that there is no sea spray mass above this height. 

\mpar
One could argue that a different quantile - or perhaps a completely unrelated measure of atmospheric conditions - represents the height of the marine boundary layer more appropriately than $z_{100};$ for example  $z_{50},$ where half the mass lies below and half lies above, or $z_{90}$ or $z_{95},$ where 10\% or 5\% of the mass lies above. The choice of $z_{100}$ is the most stringent choice and it is significantly easier to calculate that the other quantiles.

\mpar
In general, quantiles are a joint  function of the triple $(\mathcal{F}, z, r).$ But in typical cases the height of the top quantile is determined by $z$ and $r$ alone, as the following result shows.

\begin{corollary} Suppose, $\mathcal{F}(dr)=\mathcal{F}(r) dr$ and $ \mathcal{F}(r)$ is strictly positive on a finite interval $[r_{min},r_{max}]$ and equal to zero otherwise. Suppose $r\to\zeta(r)$ is continuous on $[r_{min},r_{max}]$ and $(t,r)\to (r_t,z_t(r))$ is continuous on,  

$$
A=\{\,(t,r)\mid0\leq t\leq \zeta(r),\,\,  r_{min}\leq r\leq r_{max},\,\}.
$$

\mpar
Finally, suppose $m$ is a continuous function of $z.$ Then,
$$
z_{100}=\max_{(r,t)\in A} z_t(r).
$$
\begin{proof}
Let $z^{**}=\max_{(t,r)\in A} z_t(r).$ Note that  $z_{100}\leq z^{**}$ since no droplet ever exceeds this height. Suppose, on the other hand, that $w< z^{**}.$ Since $z$ is a continuous function and $A$ is a closed set, $z$ attains its maximum value, hence  $z_t(r)>w$  for all $(t,r)$ in an open subset of $A.$ In particular, there is an interval $[a,b]\subset [r_{min}, r_{max}]$ such that,
$$
\int_0^\infty r_t^31_{(0,w]}(z_t(r))dt  < \int_0^\infty r^3_t1_{(0,z_{100}]}(z_t(r))dt
$$
for all $r\in [a,b]$. Thus,
\begin{align*}
m(w) & = \tfrac{4}{3}\pi\rho   \int_0^\infty \negthickspace \int_0^\infty r^3_s\, 1_{(0,w]}(z_s(r))\mathcal{F}(dr) ds\\
& <\tfrac{4}{3}\pi\rho   \int_0^\infty \negthickspace \int_0^\infty r^3_s\, 1_{(0,z_{100}]}(z_s(r))\mathcal{F}(dr) ds\\
& = m(z_{100}),
\end{align*}
since $\mathcal{F}(r)$ is strictly positive on $[a,b].$ By continuity this implies,
$$
m(z^{**})=\lim_{z\to z^{**}}m(z)\leq m(z_{100}),
$$
and because $m$ is non-decreasing, we have $z^{**}\leq z_{100}$ which proves the result.
\end{proof}
\end{corollary}
\mpar
\section{Examples}
It's helpful to consider examples of increasing physical accuracy. This allows us to see the mathematical consequences of various physical assumptions and, conversely, to see the need for a combination of physical assumptions in order to achieve a satisfactory model.

\mpar
We'll see that in each example the droplet's trajectory can be calculated explicitly as a function of physical parameters, which is part of the model's appeal. The calculations are elementary but lengthy so we refer them to appendices so as not to disrupt the narrative.

\begin{example} (Gravity)
The purpose of this example is to show that the mean, covariance, and quantiles can be computed exactly in the simplest case. We'll see that even in this case the formulas can be quite involved which indicates the need for a numerical approach in more realistic situations. See Appendix A for the calculations that support the assertions below.

\mpar
Let the sea spray generation function  be the indicator function of the interval $[r_{min}, r_{max}],$
$$
\mathcal{F}(r) =
\begin{cases}
1, &\text{$r_{min}\leq r\leq r_{max},$}\\
0, &\text{otherwise.}
\end{cases}
$$ 
We assume the only force acting on the droplet is the force of gravity. Then the droplet's radius is constant and the trajectory function corresponds to a freely falling particle with initial position $z=0$ and initial velocity $v>0.$ That is, 
$$
z_t^{\prime\prime} = -g, \quad z_0=0,\quad z_0^\prime = v>0,
$$
hence, 
$$
z_t(r)=vt-\tfrac{g}{2}t^2.
$$

\mpar
Recall that in the absence of external forces other than gravity the mass of a droplet does not affect its motion, so $z_t$ doesn't depend on $r,$ and $r$ doesn't depend on $t$ which greatly simplifies the analysis. In particular, the double integrals factor into products of two integrals.

\mpar
It's easy to see that the maximum height of $z_t$ is $z^*= v^2/2g$ and the duration of flight, namely the positive solution of the equation $z_t=0$, is $t^*= 2v/g.$ 

\mpar
We're interested in the set of times such that  $z_t\leq z$ or, equivalently, the roots of the equation $z_t=z.$ For every $z< z^*$ there are two roots,
$$
t^\pm_z=\frac{1}{g}\left[ v \pm\sqrt{v^2-2gz}\,\right].
$$
and for $z> z^*$ there are none. 

\mpar
To state the formulas succinctly we require additional notation. Let $a\land b$ denote the smaller of $a$ and $b,$ let $a\lor b$ denote the larger of $a$ and $b,$ and let $a_+$ denote the positive part of $a,$
$$
a_+ = \begin{cases}
a, &\text{if $a> 0,$}\\
0, &\text{if $a\leq0$}
\end{cases}
$$

\mpar
The mean and equal time covariance functions are,

\begin{align*}
m(z) & = \tfrac{2\pi\rho}{3g}\,(r_{max}^4-r_{min}^4)\left[v-1_{(0,\frac{v^2}{2g}]}(z)\sqrt{v^2-2gz}\,\right]\\
\sigma^2(t,z; t,w)  &=\tfrac{32\,\pi^2\rho^2}{63\, g}\left(r_{max}^7-r_{min}^7\right)\left[v-1_{(0,\frac{v^2}{2g}]}(z\land w) \sqrt{v^2-2g (z\land w)}\,\right].
\end{align*}

\mpar\mpar
The formula for the general covariance function $\sigma^2(z,t;w,s)$ is stated in cases accordingly as both $z$ and $w$ are greater than $z^*$, they  straddle $z^*,$ or both are less than $z^*,$ that is,

\begin{align*}
& z\land w\geq z^* && \text{(Case I)}\\
& z\land w \leq z^*\leq z\lor w && \text{(Case II)}\\
& z\lor w \leq z^* && \text{(Case III).} 
\end{align*}

\mpar
We have,
$$
\sigma^2(t,z; s,w)  = \tfrac{16}{63}\,\pi^2\rho^2\,(r_{max}^7-r_{min}^7)\Phi(|t-s|, z, w),
$$

\mpar
where in Case I,
$$
\Phi(\lambda, z,w) = (t^*-\lambda)_+,
$$

in Case II,
$$
\Phi(\lambda, z,w)  = (t^*-\lambda)_+ \land t^-_{z\land w} +\left[ (t^*-\lambda)_+-t^+_{z\land w}\right]_+,
$$
and in Case III,

\begin{align*}
\Phi(\lambda, z,w) =t^-_{z\land w}\,\land \,\left(t^-_{z\lor w}-\lambda\right)_+&+\left[ t^-_{z\land w}\,\land\, \left(t^*-\lambda\right)_+ - \left(t^+_{z\lor w}-\lambda\right)_+ \right]_+\\
&+\left[ \left(t^*-\lambda\right)_+ -t^+_{z\land w} \lor \left(t^-_{z\lor w}-\lambda\right)_+ \right]_+.
\end{align*}

\mpar
Observe that these formulas can be written out explicitly in terms of physical constants, if one wishes. For example,
$$
t^+_{z\land w}=\tfrac{1}{g}\left(v+\sqrt{v^2-2g(z\land w)}\right)
$$
and
$$
 (t^*-\lambda)_+=\begin{cases}
\tfrac{v}{2g}-\lambda, &\text{if $\lambda < \tfrac{v}{2g}$}\\
0, &\text{if $\lambda\geq\tfrac{v}{2g}$}.
 \end{cases}
$$

\mpar
Observe also that the height of the marine boundary layer can be calculated exactly,
$$
z_{MBL}=z_{100}=v^2/2g.
$$
The formula for $m(z)$ is simple enough that the equation,
$$
\frac{m(z_X)}{m(z_{100})}=\frac{X}{100}
$$
can be solved explicitly, yielding an exact formula for the quantiles, 
$$
z_X=\tfrac{v^2}{2g}\left(1-(1-\tfrac{X}{100})^2\right).
$$

We emphasize that  the point of this example is not the exact formulas for the mean, covariance, and quantiles, as interesting as they are, but rather the fact that exact formulas can be derived in the simplest case.  
\end{example}

\mpar
\begin{example} (Gravity and Stokes drag forces)
Let the sea spray generation function $\mathcal{F}(r)$ be any positive, continuous function on the interval $[r_{min},r_{max}].$ Let the trajectory function correspond to the solution of the initial value problem,

$$
\tfrac{4}{3}\pi r^3 \rho\, z_t^{\prime\prime}(r) = -\tfrac{4}{3}\pi r^3 \rho\, g +6\pi\mu r (w-z_t^\prime(r)),\quad z_0(r)=0,\quad z_0^\prime(r,w) =v> 0,
$$

\mpar
where $\rho$ is the mass density of sea water, $g$ is the force of gravity, $\mu$ is the dynamical viscosity of the atmosphere, $w$ is the vertical wind speed, and $v$ is the initial velocity. 

\mpar
The purpose of this example is to demonstrate certain unphysical effects whose presence suggests important physical processes have not been taken into account.  While this has the feel of ``proving a negative'' we think the example provides insight and the associated calculations are informative - see Appendix B.

\mpar
We show there are \textit{subcritical, critical} and \textit{supercritical} regimes depending on the relationship of a wind-speed-dependent critical radius $r_{cr}(w)$ to $[r_{min}, r_{max}]$. We show that for fixed $w$ the function,
$$
r\to\int_0^\infty 1_{(0,z]}(z_t(r))dt
$$ 
has a singularity at $r_{cr}(w)$ for sufficiently large $z.$ The singularity is not integrable with respect to $\mathcal{F}(r) dr$ which, roughly speaking, determines the aggregate behavior of the mean $m(z).$ We show,

\mpar
\begin{itemize}
\item{In the subcritical regime, every droplet returns to the ocean surface in finite time and the maximum droplet height is a bounded function of $r.$ This shows $m(z)$ is a bounded function of $z,$ which implies a finite marine boundary layer}\\
\item{In the critical regime, $m(z)$ is infinite for all sufficiently large values of $z$ due to the presence of a layer of permanently suspended droplets}\\
\item{In the supercritical regime, $m(z)$ is finite for all values of $z,$ but is an unbounded function of $z,$ which implies an infinite marine boundary layer}
\end{itemize}

\mpar
It's helpful to rewrite the equation as,
$$
z_t^{\prime\prime}=-g +\tfrac{1}{\tau}(w-z_t^\prime),
$$
where,
$$
\tau=\tau(r) =\frac{2\rho r^2}{9\mu}. 
$$

The trajectory $z_t$ has a closed form solution given by the method of integrating factors, namely,
$$
z_t =(w-g \tau)t-\tau (w-g\tau -v)\left[1-e^{-t/\tau}\right].
$$

Observe that,
\begin{align*}
z_t^\prime & = (w-g \tau)-(w-g\tau -v)e^{-t/\tau}\\
z_t^{\prime\prime} & =\tfrac{1}{\tau} (w-g\tau -v)e^{-t/\tau}
\end{align*}

so it's easy to check that $z_t$ solves the initial value problem.
\mpar
From now on we'll think of radius $r$ and wind speed $w$ as free variables and consider all other parameters as fixed. As a reminder, we write $z_t=z_t(r,w)$ for the trajectory and think of $r$ and $w$ as jointly varying.

\mpar
The first step is to observe that the solution $z_t(r,w)$ changes behavior as the expression $w-g \tau$ changes sign. Note that $w= g\tau$ if and only if,
$$
\frac{w}{r^2}=\frac{2\rho g}{9\mu}.
$$
so this condition depends on the ratio of  wind speed to the squared radius of the droplet or, equivalently, the ratio of wind speed to surface area of the droplet. Thus the equation $w-g\tau=0$ defines a critical radius as a function of wind speed, namely,
$$
r_{cr}(w)=\sqrt{\frac{9\mu w}{2\rho g}}
$$

\mpar
Evidently, if $r< r_{cr}(w)$ then $w-g\tau >0,$  which corresponds to the supercritical regime and if $r>r_{cr}(w)$ then $w-g\tau <0,$ which corresponds to the subcritical regime. Thus, there are three cases of interest: 

\begin{align*}
& \quad r_{cr}(w)< r_{min}\quad && \text{(Subcritical)}\\
& \quad r_{min}\leq r_{cr}(w)\leq r_{max} \quad &&\text{(Critical)} \\
& \quad r_{max}<  r_{cr}(w)\quad  &&\text{(Supercritical)}
\end{align*}

\mpar
and their order in terms of increasing wind speed is first subcritical then critical then supercritical. 

\mpar
\textbf{Supercritical case} In the supercritical case $w-g\tau >0$ for all $r_{min}\leq r\leq r_{max}$ and we'll see that droplets behave ballistically: they are advected by the wind and their height increases linearly without bound. Our goal is the show that $m(z)$ is finite for all $z$ but unbounded, meaning $\lim_{z\to\infty}m(z)=\infty.$ 

\mpar
In Appendix B we show that the trajectory is strictly increasing and bounded above and below by linear functions,
$$
A t \leq z_t(r,w) \leq Bt,
$$
where the constants $A$ and $B$ do not depend on $w$ or $r$. Since the upper function reaches a given height $z$ before the trajectory does, and the lower function reaches that level after the trajectory does, if follows that,
$$
\frac{z}{B}\leq\int_0^\infty 1_{(0,z]}(z_t(r,w)) dt\leq\frac{z}{A}.
$$
Substituting these inequalities into the formula for the mean we find,
$$
 \frac{z}{B}\int_{r_{min}}^{r_{max}}\mathcal{F}(r) dr\leq m(z)\leq \frac{z}{A}\int_{r_{min}}^{r_{max}}\mathcal{F}(r) dr.
$$
Observe the upper bound is finite for all $z$ and the lower bound increases to infinity as $z\to\infty,$ which gives the desired conclusion.

\mpar
\textbf{Subcritical case} In the subcritical case, $w-g\tau<0$ for all $r_{min}\leq r\leq r_{max}$ and we'll see that droplets behave similarly to Example 4.1 in that they reach a maximum height and return to the ocean surface in finite time. 

\mpar
Our goal is to prove that $m(z)$ is a bounded function of $z.$ In the appendix we show that  the duration of the droplet's flight, $t^*(r,w),$ and it's maximum height, $z^*(r,w),$ are bounded functions of $r.$ To see this is sufficient, suppose that, 
$$
t^*(r,w) \leq A\quad\text{and}\quad z^*(r,w) \leq B
$$
for all $r\in [r_{min}, r_{max}].$ Then,
$$
m(z)=\int_0^\infty\int_0^\infty 1_{(0,z]}(z_t(r,w))\mathcal{F}(r) dr\, dt\leq A\int_{r_{min}}^{r_{max}}\mathcal{F}(r) dr <\infty.
$$

Since no droplet achieves a height greater than $B$ it follows that $m(z)=m(B)$ for all $z\geq B,$ hence $m$ is bounded.

\mpar
\textbf{Critical case} In the critical case, $w-g\tau$ changes sign at the critical radius $r_{mi n}\leq r_{cr(w)}\leq r_{max}.$ The first step is to look at the trajectory $z_t(r,w)$ when $w-g\tau=0,$ that is when $r=r_{cr}(w).$ In this case the trajectory simplifies to,
$$
z_t(r, w)=v\tau\left[1-e^{- t/\tau}\right],
$$
hence $\lim_{t\to\infty}z_t(r,w)=v \tau.$ Thus, droplets of radius $r_{cr}(w)$ remain suspended in the atmosphere for all time,  reaching an asymptotic height $v\tau$. 

\mpar
Observe that the equation $z_t(r,w)=v\tau\left[1-e^{-t/\tau}\right] = z$ has a solution if and only if $z<v\tau.$ By a simple calculation, the time $t^-_z(r,w)$ at which the trajectory reaches height $z$ is,
$$
t^-_z(r,w)=\tau\log\left(\frac{1}{1-\frac{z}{v\tau}}\right).
$$

Since $z_t$ is increasing, the total amount of time the droplet spends below height $z$ is,
$$
\int_0^\infty 1_{(0,z]}(z_t(r,w))dt = \begin{cases}
t^-_z(r,w), & \text{if $z< v\tau,$}\\
\infty, & \text{if $z\geq v\tau.$}
\end{cases}
$$
Thus the time integral has a singularity at the critical radius $r_{cr}(w)$ for sufficiently large $z.$ 

\mpar
However even if $r_{min}\leq r_{cr}(w) \leq r_{max},$ the existence of the singularity by itself is not enough to conclude the mean function is infinite, as that depends on whether the singularity is integrable with respect to $\mathcal{F}(r) dr.$ Integrability, in turn, depends on the rate at which the time integral diverges as $r$ varies in the neighborhood of $r_{cr}(w).$ 

\mpar
In the appendix we show that there is a constant $B=B(v,w)$ such that for $z>B$ there exists $\epsilon >0$ and $C(z,w)>0$ such that for all $|r-r_{cr}(w)|<\epsilon,$
$$
\int_0^\infty 1_{(0,z]}(z_t(r,w))dt \geq\frac{C(z,w)}{|r-r_{cr}(w)|}.
$$
Since the function $|x|^{-1}$ is not integrable near zero, this estimate shows that for sufficiently large $z,$
$$
m(z)=\int_0^\infty\int_0^\infty 1_{(0,z]}(z_t(r,w))\mathcal{F}(r) drdt =\infty
$$
in the critical case where $r_{min}\leq r_{cr}(w)\leq r_{max}.$
\end{example}

\mpar\mpar
\begin{example} (Gravity, Stokes drag forces,  and Langmuir's $D$-squared law of evaporation) 
In the previous example, the nonintegrable singularity of the time integral at the critical radius was due to a layer of permanently suspended droplets leading, in the critical regime, to an to an infinite mean function for sufficiently large $z$. Also, in the supercritical regime, droplets behave ballistically and ascend indefinitely at a constant rate leading to an infinite marine boundary layer. These mathematical artifacts indicate the absence of crucial physical effects that must be taken into account in a satisfactory model of sea spray.

\mpar
In typical conditions, heat and mass transfer effects change the radius of droplets as they evaporate. The simplest way to take evaporation into account is through the classic $D$-squared law which states the squared diameter of the droplet, or equivalently, the square of its radius, decreases linearly in time. Specifically, for a droplet with initial radius $r,$ its time-dependent radius satisfies the equation,
$$
r^2_t =r^2-Kt, \,\, 0\leq t\leq r^2/K,
$$
where $K$ is a free parameter called the evaporation constant. Note that each droplet has a finite lifetime $r^2/K$ which effectively eliminates the mathematical artifacts mentioned above.

\mpar
Langmuir's derivation of the $D$-squared  law \cite{L} depends on a number of assumptions including that the droplet is surrounded by a quiescent medium. That is clearly not the case in our model, at least while Stokes drag forces are in effect. Nevertheless we adopt this law for two reasons. First, it allows us to calculate the physical trajectory explicitly. Second, droplets in the ballistic regime - as are all droplets of sufficiently small radius - have negligible net motion relative to the atmosphere and so are effectively in a quiesent medium. 

\mpar
Formally speaking, the dynamical trajectory of a droplet is now a system of differential equations,
\begin{align*}
z^{\prime\prime}_t & =-g +\tfrac{1}{\tau}(w-z^\prime), \quad z_0=0,\,\, z^\prime=v\\
2\,r_t\, r_t^\prime & = -K, \quad r_0=r,\,\, 0\leq t\leq r^2/K.\\
\end{align*}
In a more physically accurate model, the equation for the radius would depend on $z$ and its derivatives, which would then couple the equations together. But the equation for the $D$-squared law doesn't depend on $z$ and can be integrated directly, to the effect that $\tau$ becomes time-dependent,

$$
\tau=\tau_t = \frac{2\rho (r^2-Kt)}{9\mu},\quad 0\leq t\leq r^2/K
$$
\mpar
which implies $z_t$ is the solution of an equation with time dependent coefficients,
$$
z_t^{\prime\prime} = -g +\tfrac{1}{\tau_t}(w-z_t^\prime).
$$
In Appendix C we derive explicit formulas for the solution and it's derivatives on the time interval $[0,r^2/K]$ by the method of integrating factors and find that, 
\begin{align*}
z_t & = w \int_0^t(1-e^{-p_s}) ds + v\int_0^t e^{-p_s} ds -g\int_0^t e^{-p_s}\int_0^s e^{p_u}du\, ds
\end{align*}
where,
$$
p_t=\int_0^t\frac{ds}{\tau_s}, \quad 0\leq t < r^2/K.
$$
Note that the coefficients of $w, v$ and $g$ in the formulas for $z$ and $z^\prime$ are manifestly nonnegative.
\mpar
Calculating the associated integrals is elementary but the precise dependence of the integrals on physical parameters is sufficiently important that we provide all the details in the appendix. We find that, 
$$
z_t=w(t-\phi_t)+v\phi_t - g\psi_t,
$$
where,
\begin{align*}
\phi_t & = \tfrac{r^2}{(K+\tfrac{9\mu}{2\rho})}\left[1-\left(1-\tfrac{Kt}{r^2}\right)^{(K+\tfrac{9\mu}{2\rho})/K}\right]\\\\
\psi_t & = \begin{cases}
r^4\left[\tfrac{2\rho K-9\mu}{2\rho K +9\mu}\left(1-\left(1-\tfrac{Kt}{r^2}\right)^{1+\tfrac{9\mu}{2\rho K}}\right) -\left(\tfrac{1}{2}-\tfrac{9\mu}{4\rho K}\right)\left(1-\left(1-\tfrac{Kt}{r^2}\right)^2\right)\right], &\text{$K\neq\tfrac{9\mu}{2\rho}$}\\\\
\tfrac{r^4}{4K^2}\left[1-\left(1-\tfrac{Kt}{r^2}\right)^2 +2\left(1-\tfrac{Kt}{r^2}\right)^2\log\left(1-\tfrac{Kt}{r^2}\right)\right], & \text{$K=\tfrac{9\mu}{2\rho}$}.
\end{cases}
\end{align*}

\mpar\mpar
Recall that the function $z_t$ solves the equation on the whole interval $[0,r^2/K]$ but it represents the physical trajectory of a droplet only if both $r_t>0$ and $z_t>0.$ Thus the duration $\zeta$ of the droplet's flight is either the evaporation time $r^2/K$ or the smallest positive root of the equation $z_t=0$ on the interval $(0, r^2/K).$ If there is such a root then the droplet returns to the ocean's surface before evaporating and if not then the droplet evaporates in flight.  

\mpar
While expressions in the formula for $z_t$ are explicit, their parameter dependence is too involved to compute the time integrals,
$$
(r,v,w,K)\to\int_0^\infty 1_{(0,z]}(z_t)dt
$$
by hand, as functions of the parameters, let alone the subsequent integrals involving a physically relevant sea spray generation function.  However, in principle one can solve the inequalities $z_t\leq z$ numerically, which means we can also calculate the time integrals and hence the mean and covariance functions numerically.

\mpar
Rather than pursue that line of work here, we content ourselves with a few observations. Let $\mathcal{F}(r)$ be a positive, continuous function of a finite interval $[r_{min}, r_{max}].$

\begin{remark}
\textit{(Finite marine boundary layer) Observe that,
$$
\int_0^\infty 1_{(0,z]}(z_t) dt\leq \zeta\leq r^2/K,
$$
hence,
$$
m(z)\leq \frac{1}{K}\int_{r_{min}}^{r_{max}} r^2\,\mathcal{F}(r) dr,
$$
which shows $m(z)$ is bounded.}
\end{remark}

\begin{remark}
\textit{(Qualitative behavior of the droplet) Let $\mathbb{R}_+=(0,\infty)$ be the positive half line. As a function of its parameters, the duration $\zeta$ defines a $4D$ phase diagram,
$$
\mathbb{R}_+^4=\mathcal{P}_-\cup\,\mathcal{P}_+
$$
where,}
\begin{align*}
\mathcal{P}_- & =\{\, (r,v,w,K)\in\mathbb{R}_+^4\mid \zeta(r,v,w,K)<r^2/K\,\} \\
\mathcal{P}_+ & = \{\, (r,v,w,K)\in\mathbb{R}_+^4\mid\zeta(r,v,w,K)=r^2/K\,\}
\end{align*} 

\textit{Evidently, $\mathcal{P}_-$ represents the physical regime in which the droplet returns to the ocean before evaporating and $\mathcal{P}_+$ represents the physical regime in which the droplet evaporates in flight. A point $(r,v,w,K)\in\mathcal{P}_+$ that is not an interior point of $\mathcal{P}_+$ has the property that there are points $(p_n, v_n, w_n,  K_n)\in\mathcal{P}_-$ such that $\lim_{n\to\infty}p_n=p.$ A point of $\mathcal{P}_-$ that is not an interior point of $\mathcal{P}_-$ has a similar property in that it can be approached arbitrarily closely by points of $\mathcal{P}_+.$ }

\mpar
\textit{This suggests a modified phase diagram as follows. Let $\mathcal{P}^\circ_\pm$ denote the interior of the two phases and let,
$$
\mathcal{P}_0=\left(\mathcal{P}_-\setminus\mathcal{P}_-^\circ\right)\,\cup\,\left(\mathcal{P}_+\setminus\mathcal{P}_+^\circ\right)
$$
be the union of their frontiers. Then,
$$
\mathbb{R}^4_+=\mathcal{P}_-^\circ\cup\mathcal{P}_0\cup\mathcal{P}_+^\circ.
$$
Since $\mathcal{P}_\pm^\circ$ are disjoint, open sets by definition, it follows that $\mathcal{P}_0$ is a nonempty closed set.  The critical phase $\mathcal{P}_0$ marks the transition between two physical regimes.}

\mpar
In the appendix we show $\mathcal{P}_-^\circ$ and $\mathcal{P}_+^\circ$  are nonempty and $\mathcal{P}_0$ is their topological boundary. It follows that $\mathcal{P}_0$  is nonempty since $\mathbb{R}^4_+$ is connected. Specifically, we show that for every triple $(r,w,K)$ of positive numbers such that $K\neq 9\mu/2\rho,$ there exists $w^*>0$ such that $(r,v,w,K)\in\mathcal{P}_+^\circ$ for all $w>w^*$ and that for every triple of positive numbers $(v,w,K)$ such that $K\neq 9\mu/2\rho,$there exists $r^*>0$ such that $(r,v,w,K)\in\mathcal{P}^\circ_-$ for all $r>r^*.$ 
\end{remark}

\begin{remark}
\textit{(Height of the marine boundary layer in the supercritical regime) Fortunately, the formulas are simple enough that we can estimate the top quantile, $z_{100}=z^{**},$ explicitly in the supercritical regime. }

\mpar
\textit{In the appendix we show that for sufficiently large wind speeds, $z_t$ is strictly increasing so it achieves its maximum height at the evaporation time  $r^2/K$. Assuming even higher wind speeds if necessary, we show that $z_{r^2/K}\leq z_{r_{max}^2/K}$ and it follows that,
$$
z^{**}=z_{r_{max}^2/K} = w\times\left(\frac{r_{max}^2}{K} -\frac{r^2_{max}}{K+\frac{9\mu}{2\rho}}\right) +C
$$
for an explicit constant $C$ depending on $\rho, \mu, g, v$ and $K$ but not on $w.$ At large wind speeds $C$ is negligible compared to the other term so that,
$$
z_{MBL} \approx \frac{9\mu}{9\mu +2\rho K}\times w\times\frac{r_{max}^2}{K} 
$$
for all practical purposes. This approximation says that the height of the marine boundary layer is proportional to the wind speed times the evaporation time of the largest droplet, where the proportionality constant depends on the density of sea water, the dynamic viscosity of the atmosphere and the evaporation constant.}
 \end{remark}
\end{example}

\section{Conclusion and Future Work} The main contribution of this work is a simple model of sea spray that is both physically plausible and mathematically tractable.  By example, we showed that accounting for heat and mass transfer effects will play a crucial role in validating the model. The shot noise framework is flexible and adaptable to more general situations than the one considered here. 

\mpar
It would be valuable to calculate the mean, variance, and quantiles numerically in the case of Example 4.3 when the sea spray generation function has a generally accepted form, such as,
$$
\mathcal{F}(r)=\frac{V_A}{\mathcal{C}_A}\, \mathcal{A}\, r^{-(1+\zeta)}\left[ \Gamma_{inc}\left(m+\zeta, m\,\tfrac{r}{r_{max}}\right) - \Gamma_{inc}\left(m+\zeta, m\,\tfrac{r}{r_{min}}\right)\right].
$$
The problem depends on multiple parameters both in the sea spray generation function $(r_{min}, r_{max}, m, \zeta)$ in the physical trajectory $(v, w, K),$ and in the mean $(z)$, variance $(z,t,w,s),$ and quantiles $(X),$ so substantial parameter dependence will be a challenge.
\mpar
It would also be interesting to characterize the critical phase $\mathcal{P}_0$ in an effective way, although it may be of greater mathematical than physical interest.

\mpar
\section*{acknowledgements}
Our thanks are due to Mark Miller for inspiring this project and for his encouragement and support. He suggested we read the paper of Barr, Chen, and Fairall \cite{BCF} which incorporates sea spray in a fully coupled atmosphere-wave-ocean computational model. Their remarks on the role played by sea spray in tropical cyclone modeling prompted us to create the model presented here.

\mpar

\appendix
\mpar\mpar
\section{Details of Example 4.1}
\mpar
Verifying the formula for the mean $m(z)$ turns on a picture of the graph of $z_t.$ Observe that the duration of the droplet's flight is the positive root of the equation $z_t=vt-\tfrac{g}{2}t^2,$ namely $t^*=\tfrac{2v}{g}.$ The maximum height of the droplet occurs when,
$$
z_t^\prime=v-gt =0,
$$
that is, at time $t^\prime=\tfrac{v}{g},$ and the maximum height is,
$$
z^*= z_{t^\prime} =\frac{v^2}{2g}.
$$
If $z>z*$ then the equation $z_t=z$ has no solutions and if $z< z^*$ it has two solutions, namely the roots of the quadratic equation $z_t=vt-\tfrac{g}{2}t^2=z,$ 
$$
t_z^\pm=\frac{1}{g}\left[\, v\pm\sqrt{v^2-2gz}\,  \right].
$$
A moment's though reveals that the total amount of time the droplet spends below height $z$ can be deduced from Figure 1.
\begin{figure}[h]
\begin{tikzpicture}
\draw[->](-1,0) -- (5,0) node[right] {$t$};
\draw[->](0,-1) -- (0,3.3) node[above] {$z$};
\draw[smooth, samples =100, domain =0:4] plot(\x, {2.8*\x-0.7*\x^2});

\draw[fill=black] (4,0) circle (2pt);
\node at (4,-0.5) {$t^*$};

\draw[fill=black] (0,2.8) circle (2pt);
\node at (-0.5, 2.8) {$z^*$};

\draw[fill=black] (0,2.1) circle (2pt);
\node at (-0.5, 2.1) {$z$};

\draw[fill=black] (1,0) circle (2pt);
\node at (1, -0.5) {$t^-_z$};

\draw[fill=black] (3,0) circle (2pt);
\node at (3, -0.5) {$t^+_z$};
\draw[thin] (1,0) -- (1,2.1);
\draw[thin] (3,0) -- (3,2.1);
\draw[thin] (0,2.1) -- (3,2.1);
\draw[thin] (0,2.8) -- (2,2.8);
\end{tikzpicture}
\caption{If $z\geq z^*$ then the total amount of time $z_t$ spends below height $z$ is $t^*.$ On the other hand, if $z< z^*$ then the total amount of time $z_t(r)$ spends below height $z$ is $t^*-[t^+_z-t^-_z].$}
\end{figure}
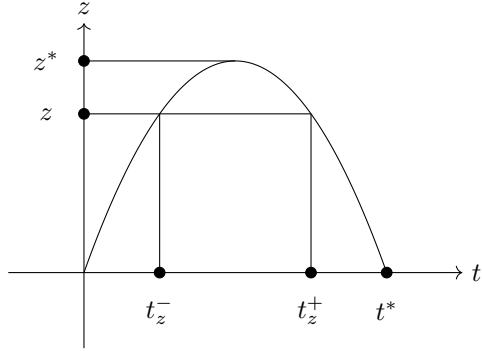

\mpar
The double integral defining $m(z)$ factors as a  product of integrals and we do the time integral first. Fixing $z$ we have,
\begin{align*}
\int_0^\infty 1_{(0,z]}(z_t) dt & = 
\begin{cases}
t^*, & \text{if $z \geq z^*$}\\
t^*-[t^+_z-t^-_z], & \text{if $z < z^*$}
\end{cases}\\
& = \begin{cases}
\tfrac{2v}{g},& \text{if $z \geq z^*$}\\
\tfrac{2v}{g} -\tfrac{2}{g}\sqrt{v^2-2gv},& \text{if $z < z^*$}
\end{cases}\\
& = \tfrac{2}{g}\left[v-1_{(0,\frac{v^2}{2g}]}(z) \sqrt{v^2-2gz}\,\right].
\end{align*}
Therefore,
\begin{align*}
m(z) & = \tfrac{4}{3}\pi\rho   \int_0^\infty \negthickspace \int_0^\infty r^3\, 1_{(0,z]}(z_t)\mathcal{F}(dr) dt\\
& = \tfrac{4}{3}\pi \rho \int_{r_{min}}^{r_{max}} r^3dr \times  \tfrac{2}{g}\left[v-1_{(0,\tfrac{v^2}{2g}]}(z) \sqrt{v^2-2gz}\,\right]\\
& = \tfrac{2\pi\rho}{3g}\left(r_{max}^4-r_{min}^4\right)\left[v-1_{(0,\tfrac{v^2}{2g}]}(z) \sqrt{v^2-2gz}\,\right].
\end{align*}

\mpar
Observe the time integral calculation also yields the formula for the equal time covariance,
\begin{align*}
\sigma^2(z,t;w,t) &=\tfrac{16}{9}\pi^2\rho^2\int_0^\infty\int_0^\infty r^6\, 1_{(0,z\land w}(z_u)\mathcal{F}(dr)\, du\\
& =\tfrac{32\,\pi^2\rho^2}{63\, g}\left(r_{max}^7-r_{min}^7\right)\left[v-1_{(0,\tfrac{v^2}{2g}]}(z\land w) \sqrt{v^2-2g (z\land w)}\,\right]
\end{align*}
\mpar
Turning to the general covariance, there are two possibilities  for $\sigma^2(z,t; s,w):$ either the smaller height occurs before the larger height or vice versa. In symbols this reads: either $w\leq z$ and $s\leq t$ or $w\leq z$ and $s\geq t.$ For the sake of argument, let's suppose $w\leq z$ and $s\leq t.$ In this case $s-s\land t + u=u$ and $t-s\land t+u =(t-s) +u$ hence,

\begin{align*}
\sigma^2(z,t; s,w) & =\tfrac{16}{9}\pi^2\rho^2 \int_0^\infty \int_0^\infty r^6\, 1_{(0, z]}(z_{t -s\land t +u})\,1_{(0, w]}(z_{s-s\land t +u})\mathcal{F}(dr) du\\
& =\tfrac{16}{9}\pi^2\rho^2 \int_{r_{min}}^{r_{max}} \int_0^\infty r^6\, 1_{(0, z]}(z_{(t -s) +u})\,1_{(0, w]}(z_{u})du dr\\
& =\tfrac{16}{63}\pi^2\rho^2 (r_{max}^7-r_{min}^7)\times \int_0^\infty 1_{(0, z]}(z_{(t -s) +u})\,1_{(0, w]}(z_{u})du.
\end{align*}

\mpar
In the time integral we think of $\lambda =t-s$ as a parameter. It varies between zero and infinity but remains constant during the integral with respect to $u.$ There are three possibilities of increasing complexity: either both $w$ and $z$ are greater than $z*,$ or they straddle $z^*,$ or both are less than $z^*.$ 

\mpar
In case both heights exceed $z^*$ we find the integrand,
$$
u\to1_{(0, z]}(z_{(t -s) +u})\,1_{(0, w]}(z_{u})
$$
is the indicator function of the set,
$$
\{\,u>0\mid u\leq t^*\,\}\cap \{\,u>0\mid \lambda +u\leq t^*\,\} = \{\,u>0\mid  \lambda +u\leq t^*\,\}
$$
whose length is $(t^*-\lambda)_+.$ Thus,
$$
\int_0^\infty 1_{(0, z]}(z_{(t -s) +u})\,1_{(0, w]}(z_{u})du = (t^*-\lambda)_+.
$$
which yields the first case. The remaining two cases follow the same argument: the time integral equals the length of a set defined by various inequalities. 

\mpar
To simplify notion, we will drop the qualification $u>0$ when specifying sets on the understanding that the qualification is in force. In addition we introduce new notation that allows us to calculate the length of an intersection of two intervals in terms of their endpoints. For example, we write,
$$
(a,b)\cap (c,d) =(a\lor c, \,b\land d),
$$
on the understanding that $(a\lor c, \,b\land d)=\emptyset$ if $a\lor c > b\lor d.$ Then the length of the intersection equals $(b\land d-a\lor c)_+$ in every case. Note that we worked with open intervals in which endpoints are excluded. But it's easy to see that the formula remains true if one or more endpoints are included because including endpoints does not change lengths. Finally, if 
$$
T=(a,b)\cup (c,d)\cup \cdots\cup (e,f)
$$
is a union of disjoint intervls and $|T|$ denotes its length then we write, 
$$
|T|=(b-a)_++(d-c)_++\cdots +(f-e)_+.
$$
\mpar
With this notation in hand, suppose $z$ and $w$ straddle $z^*.$ Then the integrand is the indicator function of the set,

\begin{align*}
\{\, \lambda +u\leq t^*\,\}\cap\{\, z_u\leq w\,\} & = \{\, \lambda +u\leq t^*\,\}\cap\left[\{\, u\leq t^-_w\,\}\cup\{ t^+_w\leq u\leq t^*\,\}\right]\\
& =(0, (t^*-\lambda)_+ \land t^-_w) \cup(t^+_w, (t^*-\lambda)_+).
\end{align*}

Observe that these two intervals are disjoint since $t^-_w< t^+_w$ hence,
\begin{align*}
\int_0^\infty 1_{(0, z]}(z_{(t -s) +u})\,1_{(0, w]}(z_{u})du & = |(0,(t^*-\lambda)_+ \land t^-_w)| + |(t^+_w, (t^*-\lambda)_+)|\\
& = (t^*-\lambda)_+ \land t^-_w +\left[ (t^*-\lambda)_+-t^+_w\right]_+.
\end{align*}

Since we assumed that $w\leq z$ or, equivalently, $w=z\land w,$ we can write the above integral more symmetrically as,
$$
 (t^*-\lambda)_+ \land t^-_{z\land w} +\left[ (t^*-\lambda)_+-t^+_{z\land w}\right]_+,
$$

which yields the second case. 

\mpar
Finally, suppose both $z$ and $w$ are less than $z^*$. This is the most complicated case. Nevertheless, the amount of time $z_u$ spends below height $w$ and simultaneously $z_{\lambda+u}$ spends below height $z$ can be deduced from Figure 2.

\begin{figure}[h]
\begin{tikzpicture}
\draw[->](-1,0) -- (5,0) node[right] {$t$};
\draw[->](0,-1) -- (0,3.3) node[above] {$z$};
\draw[smooth, samples =100, domain =0:4] plot(\x, {2.8*\x-0.7*\x^2});

\draw[fill=black] (4,0) circle (2pt);
\node at (4,-0.5) {$t^*$};

\draw[fill=black] (0,2.8) circle (2pt);
\node at (-0.45, 2.82) {$z^*$};

\draw[fill=black] (0,2.55) circle (2pt);
\node at (-0.5, 2.53) {$z$};

\draw[fill=black] (0,2.1) circle (2pt);
\node at (-0.5, 2.1) {$w$};

\draw[fill=black] (1,0) circle (2pt);
\node at (0.96, -0.5) {$t_w^-$};

\draw[fill=black] (3,0) circle (2pt);
\node at (3.1, -0.5) {$t_w^+$};

\draw[fill=black] (1.4,0) circle (2pt);
\node at (1.5, -0.5) {$t^-_z$};

\draw[fill=black] (0.3,0) circle (2pt);
\node at (0.3, -0.5) {$\lambda$};

\draw[fill=black] (2.6,0) circle (2pt);
\node at (2.55, -0.5) {$t_z^+$};
\draw[thin] (1,0) -- (1,2.1);
\draw[thin] (3,0) -- (3,2.1);
\draw[thin] (0,2.1) -- (3,2.1);
\draw[thin] (1.4,0) -- (1.4,2.55);
\draw[thin] (2.6,0) -- (2.6,2.55);
\draw[thin] (0, 2.8) -- (2, 2.8);
\draw[thin] (0,2.55) -- (2.6,2.55);
\end{tikzpicture}
\caption{In this diagram, $w< z < z^*,\, s<t,$ and $\lambda =t-s$ is small.}  
\end{figure}
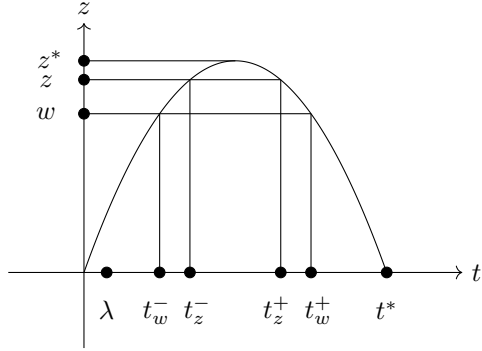

Evaluating the time integral just amounts to calculating the length of the set,

$$
T^{\lambda}_{ w, z} = \{\, 0<z_{\lambda+u}\leq z\, \} \cap \{\, 0<z_u\leq w \,\}.
$$
 
 Note that,
$$
\{\,0<z_u\leq w \,\}  = (0, t^-_w] \cup [t^+_w, t^*)
$$
and
\begin{align*}
\{\,0<z_{\lambda+u}\leq z\,\} & = \{\, 0< \lambda+u\leq t^-_z\,\} \cup \{\, t^+_z\leq \lambda +u<t^*\,\}\\
& =\left(0, (t^-_z-\lambda)_+\right]\cup \left[(t^+_z -\lambda)_+, (t^*-\lambda)_+\right).
\end{align*}

Thus we have,
\begin{align*}
T^{\lambda}_{ w, z} & =\Big[\left(0, t^-_w\right] \cup \left[t^+_w, t^*\right)\Big] \cap \Big[\left(0, (t^-_z-\lambda)_+\right] \cup  \left[(t^+_z -\lambda)_+, (t^*-\lambda)_+ \right)\Big]\\
& = I\,\cup\, II\,\cup\, III\,\cup\, IV,
\end{align*}
where,
\begin{align*}
I & = \left(0, t^-_w\right]\cap \left(0, (t^-_z-\lambda)_+\right]\\
II & =\left(0, t^-_w\right]\cap \left[(t^+_z -\lambda)_+, (t^*-\lambda)_+ \right)\\
III & = \left[t^+_w, t^*\right)\cap \left(0, (t^-_z-\lambda)_+\right]\\
IV & =\left[t^+_w, t^*\right)\cap\left[(t^+_z -\lambda)_+, (t^*-\lambda)_+ \right).
\end{align*}

First of all notice that $III$ is empty. This is because the two sets are disjoint since $t^-_z < t^+_z.$ Second, $I, II,$ and $IV$ are pairwise disjoint for similar reasons. For example, $I$ and $II$ are disjoint because $t^-_z<t^+_z$ and the others are disjoint because $t^-_w <t^+_w.$ Therefore the length of $T^\lambda_{z,w}$ equals the sum of the lengths of $I, II$ and $IV.$ Thus $|T^\lambda_{z,w}|=|I|+|II|+|IV|$ where,

\begin{align*}
|I| & =t^-_w\land\left(t^-_z-\lambda\right)_+ \\
|II| & = \Big[ t^-_w\land \left(t^*-\lambda\right)_+ - \left(t^+_z-\lambda\right)_+ \Big]_+\\
|IV| & = \Big[ \left(t^*-\lambda\right)_+ -t^+_w\lor \left(t^-_z-\lambda\right)_+ \Big]_+
\end{align*}
The formula for the general covariance follows by observing $w-z\land w$ and $z=z\lor w.$ 
\mpar
Observe that each droplet attains a maximum height of $z^*=v^2/2g.$ Thus $m(z)=m(z^*)$ for all $z>z^*$ and by inspection $m(z)<m(z^*)$ for all $z< z^*,$ which shows that $z_{100}=z^*.$ Since the formula for $m(z)$ is a product, the factors containing $r$ cancel in the quotient and we have,
$$
\frac{X}{100} = \frac{m(z_X)}{m(z_{100})} = \frac{v-\sqrt{v^2-2gz_X}}{v}
$$
hence,
$$
z_X=\tfrac{v^2}{2g}\left(1-(1-\tfrac{X}{100})^2\right).
$$

\mpar\mpar
\section{Details of Example 4.2}

Recall that,
\begin{align*}
z_t & =(w-g \tau)t-\tau (w-g\tau -v)\left[1-e^{-t/\tau}\right],\\
z_t^\prime & = (w-g \tau)-(w-g\tau -v)e^{-t/\tau}\\
z_t^{\prime\prime} & =\tfrac{1}{\tau} (w-g\tau -v)e^{-t/\tau},
\end{align*}
and the the critical radius,
$$
r_{cr}(w)=\sqrt{\frac{9\mu w}{2\rho g}}
$$
is characterized by the fact that $w-g\tau(r)>0,$ if $r<r_{cr(w)},$ and $w-g\tau(r)<0,$ if $r>r_{cr(w)}.$

\mpar
\textbf{Supercritical case.} If $r_{max}<r_{cr(w)}$ then all droplets with radii in the interval $[r_{min}, r_{max}]$ behave ballistically meaning,
$$
\lim_{t\to\infty}\frac{z_t(r)}{t}=w-g\tau(r) >0,
$$
hence droplet height increases linearly without bound.

\mpar
Estimating the growth of $z_t$ independently of $r$ is based on the fact it lies strictly between the graphs of the linear functions $f(t)=vt$ and $g(t)=(w-g\tau)t.$ Obviously, the linear function with the larger slope lies above while the one with the smaller slope lies below, namely,
$$
v\land (w-g\tau(r)\times t\,\,\leq z_t(r)\leq \,\,v\lor (w-g\tau(r))\times t.
$$

Observe that,
$$
v\lor (w-g\tau(r))\geq v\lor(w-g\tau(r_{min})) = A
$$
and
$$
v\land (w-g\tau(r))\leq v\land(w-g\tau(r_{max})) =B.
$$
It follows that,
$$
At\leq z_t(r,w)\leq Bt
$$ 
for all $r\in [r_{min}, r_{max}].$
\mpar
\textbf{Subcritical case.} If $r_{cr}(w) < r_{min}$ then $w- g\tau(r) <0$ for all radii in the interval $[r_{min}, r_{max}]$ and therefore the droplet behaves similarly to Example 4.1. Our goal is to prove that the duration $t^*$ and the maximum height $z^*$ are bounded independently of $r.$  

\mpar
Evidently, $\lim_{t\to\infty}z_t=-\infty,$ so the smallest root $t=t^*(r,w)$ of the equation,
$$
(w-g\tau)t -\tau(w-g\tau-v)\left[1-e^{-t/\tau}\right] =0
$$

is the duration of the droplet's flight. This happens when,
$$
t^*=\tau(r)\times \left(1+\frac{v}{g\tau(r)-w}\right)\times \left[1- e^{-t^*/\tau(r)}\right],
$$
which is bounded above by,
$$
A=\tau(r_{max})\times \left(1+\frac{v}{g\tau(r_{min})-w}\right).
$$
\mpar
To find the maximum height observe that $z_t^\prime(r,w)=0$ if and only if 
$$
w-g\tau = (w-g\tau -v)e^{-t/\tau}
$$
which occurs at the time,
$$
t^\prime(r,w)=\tau\log\left(1+\frac{v}{g\tau -w}\right).
$$

Since $z_t^{\prime\prime}<0$ for all $t>0$ it follows that the droplet reaches its maximum height $z^*=z^*(r,w)$ at time $t^\prime=t^\prime(r,w)$ where,
\begin{align*}
z^* & = z_{t^\prime}\\
& = (w-g\tau)t^\prime -\tau(w-g\tau-v)\left[1-e^{-t^\prime/\tau}\right]\\
& = -\tau(g\tau-w)\log\left(1+\tfrac{v}{g\tau -w}\right)+\tau(g\tau -w +v)\left[1-\frac{1}{1+\tfrac{v}{g\tau -w}}\right].
\end{align*}

Observe that the first term above is negative hence,
\begin{align*}
z^* &\leq \tau(r)(g\tau(r) -w +v)\left[1-\frac{1}{1+\tfrac{v}{g\tau(r) -w}}\right]\\
&\leq \tau(r)(g\tau(r) -w +v)\\
&\leq \tau(r_{max})(g\tau(r_{max})-w +v) = B,
\end{align*}
so $z^*$ is bounded above by a fixed constant.

\mpar
\textbf{Critical case.} Our goal is to show that for $z>v\tau(r_{cr}(w))$ there exists $\epsilon >0$ and $C(z,w)>0$ such that for all $|r-r_{cr}(w)|<\epsilon,$
$$
\int_0^\infty 1_{(0,z]}(z_t(r,w))dt \geq\frac{C(z,w)}{|r-r_{cr}(w)|}.
$$
We do this by estimating the time integral as $r$ increases to $r_{cr}(w)$ and again as $r$ decreases to $r_{cr},$ as different arguments are required depending on the sign of $w-g\tau(r).$

\mpar
First, suppose $r<r_{cr}$ so that $w-g\tau(r)$ is positive but decreases to zero as $r$ increases to $r_{cr}(w).$ This is the supercritical case. Note that $z^{\prime\prime}_t >0$ for all $t>0,$ if $v<w-g\tau,$ and $z^{\prime\prime}_t < 0$ for all $t>0,$ if $v >w-g\tau.$ In the former case $z_t^\prime$ increases from $v$ to $w-g\tau$ while in the latter case it decreases from $v$ to $w-g\tau.$ In either case, $z^\prime (r,w)$ is positive hence $z_t(r,w)$ is strictly increasing. Thus the equation, $z_t(r)=z$ has a unique solution $t =t^-_z(r,w)$ for every $z>0$ and it follows that,
$$  
\int_0^\infty 1_{(0,z]}(z_t(r,w))dt = t^-_z(r,w).
$$

Recall that,
$$
z_t(r)\leq v\lor (w-g\tau(r))\times t
$$
and therefore,
$$
t^-_z(r,w)\geq\frac{z}{v\lor (w-g\tau(r))}.
$$
Since $w-g\tau(r)$ decreases to zero as $r$ increases to $r_{cr}(w)$ the lower bound becomes $z/(w-g\tau(r))$ for all $r$ sufficiently close to $r_{cr}.$ Observe that,
\begin{align*}
w-g\tau(r) & = w-\tfrac{2\rho g}{9\mu}[r_{cr}^2(w) - (r_{cr}^2(w)-r^2)]\\
& =  w-g\tau(r_{cr}(w)) +\tfrac{2\rho g}{9\mu}(r_{cr}^2(w)-r^2)\\
& = \tfrac{2\rho g}{9\mu}(r_{cr}(w)-r)(r_{cr}(w)+r).
\end{align*}

It follows that,
$$
(r_{cr}(w)-r)\frac{z}{w-g\tau(r)}=\frac{z}{\frac{2\rho g}{9\mu}(r_{cr}(w)+r)}
$$
and therefore,
$$
\lim_{r\nearrow r_{cr}(w)}(r_{cr}(w)-r)t^-_z(r,w)\geq\frac{18\mu z}{4\rho g\, r_{cr}(w)}.
$$
Thus there exist $\epsilon =\epsilon(v)$ and a constant $C=C(z,w)$ such that for all $r_{cr}(w)-r<\epsilon,$
$$
t^-_z(r,w)\geq \frac{C(z,w)}{r_{cr}(w)-r}.
$$

\mpar
Now suppose $r>r_{cr}(w)$ so that $w-g\tau(r)$ is negative but decreases to zero as $r$ decreases to $r_{cr}(w).$ This is the subcritical case. 

\mpar
Recall the maximum height $z^*$ is bounded, 
$$
z^*\leq \tau(r_{max})(g\tau(r_{max})-w+v) =B
$$
hence for all $z >B,$ the total amount of time $z_t$ spends below $z$ is the duration $t^*= t^*(r,w).$ Recall that the the duration satisfies the equation,
$$
t^* = \tau(r)\times \left(1+\frac{v}{g\tau(r)-w}\right)\times \left[1-e^{-t^*/\tau(r)}\right].
$$
To bound this expression from below it suffices to bound $1-e^{-t^*/\tau(r)}$ away from zero. Let us show that there is an $\epsilon=\epsilon(v)$ such that $1-e^{-t^*/\tau(r)}\geq 1/2$ for all $r-r_{cr}(w)<\epsilon.$ Assuming this is the case we have,
$$
t^* \geq \tfrac{1}{2}\, \tau(r)\left(1+\frac{v}{g\tau(r)-w}\right)
$$
which implies by a previous argument that, there exists a possibly smaller $\epsilon=\epsilon(v)$ and a constant  $C=C(z,w)$ such that for all $z>B$ and $r-r_{cr}(w)<\epsilon,$
$$
t^-_z(r,w)=t^*_z(r,w)\geq \frac{C(z,w)}{r_{cr}(w)-r}.
$$

The key observation is that the duration is larger than the time it takes the trajectory to hit its maximum, that is to say,
$$
t^*(r,w)\geq t^\prime_z(r,w) = \tau(r)\log\left(1+\frac{v}{g\tau(r)-w}\right).
$$
Since $g\tau(r)-w$ decreases to zero as $r$ decreases to $r_{cr}(w),$ there is an $\epsilon=\epsilon(v)$ such that $g\tau(r)-w\leq v$ for all $r-r_{cr}(w)<\epsilon.$ In this case, $t^*\geq \tau\log(2)$ and therefore,
$$
1-e^{-t^*/\tau}\geq 1-e^{-\log(2)} = \tfrac{1}{2}.
$$

\mpar
\section{Details of Example 4.3} 
Our first goal is to derive explicit formulas for the solution,
$$
z_t^{\prime\prime} = -g +\tfrac{1}{\tau_t}(w-z_t^\prime)'\quad z_0=0,\, z^\prime_0=v
$$
where,
$$
\tau_t=\frac{9\mu}{2\rho(r^2-Kt)}
$$
by the method of integrating factors. Let,
$$
p_t=\int_0^t\frac{ds}{\tau_s},\,\, 0\leq t<\frac{r^2}{K}.
$$
Then,
\begin{align*}
\left(e^{p_t}z^\prime_t\right)^\prime & = e^{p_t}\left(z^{\prime\prime} +p^\prime_tz^\prime_t\right)\\
& = e^{p(t)}\left(z^{\prime\prime} +\tfrac{1}{\tau_t}z^\prime_t\right)\\
& = e^{p(t)}\left(-g+\tfrac{w}{\tau_t}\right)\\
& = e^{p(t)}\left(-g+w p^\prime_t\right)\\
\end{align*}

Thus,
\begin{align*}
e^{p_t} z^\prime_t -v  & = w\int_0^t p^\prime_s e^{p_s}ds -g\int_0^t e^{p_s}ds\\
& = w [e^{p_t} -1] - g\int_0^t e^{p_s}ds
\end{align*}
and it follows that,
$$
z^\prime_t =w(1-e^{-p_t}) + v\, e^{-p_t} - g\,e^{-p_t}\int_0^t e^{p_s}ds.
$$
Integrating once more we find,
\begin{align*}
z_t &= wt + (v-w)\int_0^t e^{-p_s} ds -g\int_0^t e^{-p_s}\int_0^s e^{p_u}du\, ds\\
& = w(t-\phi_t) + v\phi_t - g\psi_t.
\end{align*}

Calculating the integrals is straightforward. We have,
\begin{align*}
p_t & = \int_0^t\frac{ds}{\tau_s}\\
& =\tfrac{9\mu}{2\rho}\int_0^t \frac{ds}{r^2-Ks}\\
& = \tfrac{9\mu}{2\rho K}\int_{r^2-Kt}^{r^2} \frac{du}{u}\\
& = \tfrac{9\mu}{2\rho K} \log\left(\frac{1}{1-\frac{Kt}{r^2}}\right).
\end{align*}
hence,
$$
e^{p_t} = \left(1-\frac{Kt}{r^2}\right)^{-\frac{9\mu}{2\rho K}} \quad\text{and}\quad e^{-p_t} = \left(1-\frac{Kt}{r^2}\right)^{\frac{9\mu}{2\rho K}}.
$$
With these formulas in hand we find,
\begin{align*}
\phi_t & = \int_0^te^{-p_s} ds\\
& = \int_0^t\left(1-\tfrac{Ks}{r^2}\right)^ads,\quad a=\tfrac{9\mu}{2\rho K},\\
& =\tfrac{r^2}{K}\int_{1-\frac{Kt}{r^2}}^1 u^a du\\
& = \tfrac{r^2}{K(1+a)}\left[1-\left(1-\tfrac{Kt}{r^2}\right)^{1+a}\right]\\
& = \tfrac{r^2}{(K+\tfrac{9\mu}{2\rho})}\left[1-\left(1-\tfrac{Kt}{r^2}\right)^{(K+\tfrac{9\mu}{2\rho})/K}\right],
\end{align*}

Similarly we have,
\begin{align*}
\int_0^t e^{p_s} ds & = \int_0^t\left(1-\tfrac{Ks}{r^2}\right)^{-a} ds,\quad a=\tfrac{9\mu}{2\rho K}\\
& = \tfrac{r^2}{K}\int_{1-\frac{Kt}{r^2}}^1 u^{-a} du\\
& = \begin{cases}
\tfrac{r^2}{(K -\tfrac{9\mu}{2\rho})}\left[1-\left(1-\frac{Kt}{r^2}\right)^{(K-\tfrac{9\mu}{2\rho})/K}\right], &\text{$K\neq \tfrac{9\mu}{2\rho},$}\\\\
\tfrac{r^2}{K}\log(\frac{1}{1-\frac{Kt}{r^2}}), &\text{$K=\tfrac{9\mu}{2\rho}$}
\end{cases}
\end{align*}

\mpar
It follows that,
\begin{align*}
e^{-p_t}\int _0^te^{p_s}ds & = \begin{cases}
\tfrac{r^2}{(K-\tfrac{9\mu}{2\rho})}\left[\left(1-\frac{Kt}{r^2}\right)^{\frac{9\mu}{2\rho K}}-\left(1-\frac{Kt}{r^2}\right)\right], &\text{$K\neq\frac{9\mu}{2\rho}$}\\\\
\tfrac{r^2}{K}\left(1-\frac{Kt}{r^2}\right)\log\left(\frac{1}{1-\frac{Kt}{r^2}}\right), &\text{$K=\frac{9\mu}{2\rho}$}
\end{cases}\\
\end{align*}

\mpar
To exhibit a formula for $\psi_t$ we argue by cases. If $K\neq 9\mu/2\rho$ we have,
\begin{align*}
\psi_t & = \int_0^t e^{-p_s}\int_0^s e^{p_u} du\, ds\\
 & = \tfrac{r^2}{(K-\tfrac{9\mu}{2\rho})}\int_0^t \left[\left(1-\tfrac{Ks}{r^2}\right)^{\frac{9\mu}{2\rho K}}-\left(1-\tfrac{Ks}{r^2}\right)\right] ds\\
& = (K-\tfrac{9\mu}{2\rho}) r^2\left[\tfrac{r^2}{K+\frac{9\mu}{2\rho }}\left(1-\left(1-\tfrac{Kt}{r^2}\right)^{1+\tfrac{9\mu}{2\rho K}}\right) - \tfrac{r^2}{2K}\left(1-\left(1-\tfrac{Kt}{r^2}\right)^2\right)\right]\\
& = r^4\left[\tfrac{2\rho K-9\mu}{2\rho K +9\mu}\left(1-\left(1-\tfrac{Kt}{r^2}\right)^{1+\tfrac{9\mu}{2\rho K}}\right) -\left(\tfrac{1}{2}-\tfrac{9\mu}{4\rho K}\right)\left(1-\left(1-\tfrac{Kt}{r^2}\right)^2\right)\right].
\end{align*}

On the other hand, if $K=9\mu/2\rho$ then,
\begin{align*}
\psi_t & = \int_0^t e^{-p_s}\int_0^s e^{p_u} du\, ds\\
& = \tfrac{r^2}{K}\int_0^t \left(1-\tfrac{Ks}{r^2}\right)\log\left(\tfrac{1}{1-\tfrac{Ks}{r^2}}\right)\\
& = -\tfrac{r^4}{K^2}\int_{1-\tfrac{Kt}{r^2}}^1 u\log(u)du\\
& =\tfrac{r^4}{K^2}\left[\tfrac{1}{4}u^2 -\tfrac{1}{2}u^2\log(u) \right]_{1-\tfrac{Kt}{r^2}}^1\\
& =\tfrac{r^4}{4K^2}\left[1-\left(1-\tfrac{Kt}{r^2}\right)^2 +2\left(1-\tfrac{Kt}{r^2}\right)^2\log\left(1-\tfrac{Kt}{r^2}\right)\right].
\end{align*}

\mpar
With these formulas in hand, it follows that,
$$
z_t^\prime =w(1-\phi^\prime_t)+ v\phi^\prime_t  - g\psi^\prime_t,
$$
where,
\begin{align*}
\phi_t^\prime & =e^{-pt} = \left(1-\frac{Kt}{r^2}\right)^{\frac{9\mu}{2\rho K}}\\\\
\psi_t^\prime & = e^{-p_t}\int_0^t e^{p_s} ds = \begin{cases}
\tfrac{r^2}{(K-\tfrac{9\mu}{2\rho})}\left[\left(1-\frac{Kt}{r^2}\right)^{\frac{9\mu}{2\rho K}}-\left(1-\frac{Kt}{r^2}\right)\right], &\text{$K\neq\frac{9\mu}{2\rho}$}\\\\
\tfrac{r^2}{K}\left(1-\frac{Kt}{r^2}\right)\log\left(\frac{1}{1-\frac{Kt}{r^2}}\right), &\text{$K=\frac{9\mu}{2\rho}$}
\end{cases}
\end{align*}

In addition we have,
$$
z_t = w(t-\phi_t) + v\phi_t - g\psi_t
$$
where,
\begin{align*}
\phi_t & = \tfrac{r^2}{(K+\tfrac{9\mu}{2\rho})}\left[1-\left(1-\tfrac{Kt}{r^2}\right)^{(K+\tfrac{9\mu}{2\rho})/K}\right]\\
\psi_t & = \begin{cases}
r^4\left[\tfrac{2\rho K-9\mu}{2\rho K +9\mu}\left(1-\left(1-\tfrac{Kt}{r^2}\right)^{1+\tfrac{9\mu}{2\rho K}}\right) -\left(\tfrac{1}{2}-\tfrac{9\mu}{4\rho K}\right)\left(1-\left(1-\tfrac{Kt}{r^2}\right)^2\right)\right], &\text{$K\neq\tfrac{9\mu}{2\rho}$}\\\\
\tfrac{r^4}{4K^2}\left[1-\left(1-\tfrac{Kt}{r^2}\right)^2 +2\left(1-\tfrac{Kt}{r^2}\right)^2\log\left(1-\tfrac{Kt}{r^2}\right)\right], & \text{$K=\tfrac{9\mu}{2\rho}$}.
\end{cases}
\end{align*}

\mpar
Our next goal is to show that the two phases $\mathcal{P}_\pm^\circ$ are nonempty. Regarding $\mathcal{P}_-^\circ$ observe that $z_{r^2/K}=0$ if and only if,
$$
w(T-\phi_T)+v\phi_t-g\psi_T =0,\quad T=\frac{r^2}{k}.
$$
Noting that the formulas for $\phi_t$ and $\psi_t$ simplify significantly when $t=T=r^2/K,$ the equation above is equivalent to,
$$
w\left(\frac{r^2}{K}-\frac{r^2}{K+\frac{9\mu}{2\rho}}\right) +v\,\frac{r^2}{K+\frac{9\mu}{2\rho}} = g\, r^4\left[\frac{K-\frac{9\mu}{2\rho}}{K+\frac{9\mu}{2\rho}} -\left(\frac{1}{2}-\frac{9\mu}{4\rho K}\right)\right]
$$
when $K\neq 9\mu/2\rho.$ This equation has the form,
$$
ar^2=br^4
$$
for specific constants $a$ and $b$ or, equivalently, $r^2  = a/b.$ Now,
\begin{align*}
a & =w\left(\frac{1}{K}-\frac{1}{K+\frac{9\mu}{2\rho}}\right) +v\,\frac{1}{K+\frac{9\mu}{2\rho}}\\
& =\frac{w\,\frac{9\mu}{2\rho} + v\,K}{K\left(K+\frac{9\mu}{2\rho}\right)}
\end{align*}
and,
\begin{align*}
b &= g\,\left[\frac{K-\frac{9\mu}{2\rho}}{K+\frac{9\mu}{2\rho}} -\left(\frac{1}{2}-\frac{9\mu}{4\rho K}\right)\right]\\
& = g\left[\frac{K-\frac{9\mu}{2\rho}}{K+\frac{9\mu}{2\rho}}-\frac{K-\frac{9\mu}{2\rho}}{2K}\right]\\
& =g\left(K-\frac{9\mu}{2\rho}\right)\left[\frac{1}{K+\frac{9\mu}{2\rho}}-\frac{1}{2K}\right]\\
& =\frac{g\left(K-\frac{9\mu}{2\rho}\right)^2}{2K\left(K+\frac{9\mu}{2\rho}\right)}.
\end{align*}
Thus,
$$
r^2 =\frac{a}{b}= \frac{2\left(w\frac{9\mu}{2\rho} +vK\right)}{g\left(K-\frac{9\mu}{2\rho}\right)^2}.
$$
Let $r^*= r^*(v,w,K)$ be the square root of the quantity above. These calculations show that if $r> r^*$ then $z_{r^2/K}< 0,$ hence $\zeta(r,v,w,K)< r^2/K$ which implies $p=(r,v,w,K)\in\mathcal{P}_-.$ Since $z_t$ is a continuous function of its parameters, $z_{r^2/K}<0$ all points $q\in\mathcal{P}_-$ sufficiently close to $p$ which shows that $p$ is an interior point of $\mathcal{P}_-.$ Therefore $\mathcal{P}_-^\circ$ is a nonempty open set.

\mpar
Next let's consider $\mathcal{P}_+^\circ.$ Recall that,
$$
z^\prime_t=w(1-\phi_t^\prime) + v\phi_t^\prime -g\psi_t^\prime
$$
where,
\begin{align*}
\phi_t^\prime & =\left(1-\frac{Kt}{r^2}\right)\\
\psi_t^\prime & = \tfrac{r^2}{(K-\tfrac{9\mu}{2\rho})}\left[\left(1-\frac{Kt}{r^2}\right)^{\frac{9\mu}{2\rho K}}-\left(1-\frac{Kt}{r^2}\right)\right], &\text{$K\neq\frac{9\mu}{2\rho}$}
\end{align*}

\mpar
Observe that $z_t^\prime >0$ provided the graph of the function $t\to g\psi_t^\prime$ lies below the line segment $t\to (0,v)\phi_t^\prime + (r^2/K, w)(1-\phi_t^\prime).$ Evidently, parameters $p=(r,v,w,K)$ for which this is true must lie in the phase $\mathcal{P}_+$ because $z_t>0$ for all $0<t\leq r^2/K.$ Since $z_t$ is a continuous function of its parameters, it follows that $z_t$ is strictly positive on $(0,r^2/K]$ for all $q\in\mathcal{P}_+$ sufficiently close to $p,$ which shows that $p\in\mathcal{P}_+^\circ.$

\mpar
To see that this is true for sufficiently large $w,$ note that $\phi_t^\prime$ is \textit{unimodal} - that is, it increases from zero to a maximum height at then decreases back to zero. Thus there is a $t^*$ such that $\phi_t^\prime$ increases from zero on the interval $(0,t^*)$ and decreases to zero on the interval $(t^*,r^2/K).$  Observe that as $w$ increases from zero to infinity, the line segment intersects the graph in two places, then becomes tangent to the graph at a specific value $w^*,$ and is disjoint from the graph of $f$ for all $w>w^*.$ This shows that for fixed $(r,v,K)$ such that $K\neq 9\mu/2\rho$ there is a $w^*$ such that for all $w>w^*,$ $p=(r,v,w,K)\in\mathcal{P}_+^\circ.$  Therefore, $\mathcal{P}_+^\circ$ is a nonempty open set.

\mpar
Our final goal is to evaluate $z^{**}$ in the supercritical regime. We argue in the case $K\neq 9\mu/2\rho$ and observe the remaining case is similar. Note that in the supercritical phase $\mathcal{P}_+,$ the droplet does not return to the ocean surface. Therefore,
\begin{align*}
z^{**} &=\max\{\, z_t(r)\mid 0\leq t\leq r^2/K,\,r_{min}\leq r\leq r_{max}\}\\
&\geq \max\{\, z_{r^2/K}(r)\mid\,r_{min}\leq r\leq r_{max}\}
\end{align*}
since the maximum height of the droplet may occur at the evaporation time. On the other hand, we know that $z_t\leq z_{r^2/K}$ at sufficiently high wind speeds $w>w^*.$ Then we have,
\begin{align*}
z^{**} & = \max\{\, z_{r^2/K}(r)\mid\,r_{min}\leq r\leq r_{max}\}\\
& = \max\{\, ar^2-br^4\mid,r_{min}\leq r\leq r_{max}\}\\
\end{align*}
where,
$$
a=\frac{w\,\frac{9\mu}{2\rho} + v\,K}{K\left(K+\frac{9\mu}{2\rho}\right)}\quad\text{and}\quad b=\frac{g\left(K-\frac{9\mu}{2\rho}\right)^2}{2K\left(K+\frac{9\mu}{2\rho}\right)}.
$$

\mpar
Taking the derivative of $ar^2-br^4$ in $r$ we find this polynomial is increasing if and only if $2ar-4br^3>0$ or, equivalently, if $a>2br^2.$ Since $a$ increases in proportion to $w$ and $b$ does not depend on $w,$ there exists $w^{**}\geq w^*$ such that $a>2br^2$ positive on $[r_{min}, r_{max}].$ This implies $z_{r^2/K}(r)$ is maximized at $r=r_{max}$ hence,
\begin{align*}
z^{**} & =z_T(r_{max}),\quad T=r_{max}^2/K\\
& = w(T-\phi_T) + v\phi_T - g\psi_T\\
& = w\left(\frac{r_{max}^2}{K}-\frac{r_{max}^2}{K+\frac{9\mu}{2\rho}}\right) + \left(v\frac{r_{max}^2}{K+\frac{9\mu}{2\rho}} - \frac{g\left(K-\frac{9\mu}{2\rho}\right)^2}{2K\left(K+\frac{9\mu}{2\rho}\right)}\right)\\
& =\frac{9\mu}{9\mu + 2\rho K}\times w\times\frac{r_{max}^2}{K} +C.
\end{align*}

Note that $C$ depends on $\rho, \mu, g, r_{max}, v,$  and $K$ but not on $w$ so it makes a negligible contribution to $z^{**}$ at sufficiently large wind speeds.
\end{document}